\documentclass[]{siamart}

\usepackage{amsmath}
\usepackage{mathrsfs,amssymb,amsmath,booktabs,array,xcolor,url,cite,color,soul,multirow,multicol}
\usepackage{mathbbold,bbm}
\usepackage{slashbox}
\usepackage{stfloats}
\usepackage{amsfonts}
\usepackage{extarrows}
\usepackage{textcomp}
\usepackage{ulem}
\usepackage{cases}
\usepackage{bm}
\usepackage{arydshln}
\usepackage[all]{xy}

  \newcommand{\tabincell}[2]{\begin{tabular}{@{}#1@{}}#2\end{tabular}}

\title{Complex Balancing Reconstructed to the Asymptotic Stability of Mass-action Chemical Reaction Networks with Conservation Laws
}

\author{Min Ke\footnotemark[2], Zhou Fang\footnotemark[2]
	\and Chuanhou Gao\footnotemark[2]}

\begin{document}

\maketitle

\renewcommand{\thefootnote}{\fnsymbol{footnote}}

\footnotetext[2]{School of Mathematical Sciences, Zhejiang University,
	Hangzhou 310027, P. R. China. (\email{kemin@zju.edu.cn, Zhou\_Fang@zju.edu.cn,
		gaochou@zju.edu.cn (Correspondence)}).}

\renewcommand{\thefootnote}{\arabic{footnote}}

\begin{abstract}
Motivated by the fact that the {pseudo-Helmholtz function} is a valid Lyapunov function for characterizing asymptotic stability of complex balanced mass action systems (MASs), this paper develops the generalized {pseudo-Helmholtz function} for stability analysis for more general MASs {assisted with conservation laws}. The key technique is to transform the original network into two different MASs, defined by reconstruction and reverse reconstruction, with an important aspect that the dynamics of the original network for free species is equivalent to that of the reverse reconstruction. Stability analysis of the original network is then conducted based on an analysis of how stability properties are retained from the original network to the reverse reconstruction. We prove that the reverse reconstruction possesses only an equilibrium in each positive stoichiometric compatibility class if the corresponding reconstruction is complex balanced. Under this complex balanced reconstruction strategy, the asymptotic stability of the reverse reconstruction, which also applies to the original network, is thus reached by taking the generalized {pseudo-Helmholtz function} as the Lyapunov function. To facilitate applications, we further provide a systematic method for computing complex balanced reconstructions {assisted with conservation laws}. Some representative examples are presented to exhibit the validity of the complex balanced reconstruction strategy.
\end{abstract}



\section{Introduction}
Chemical reaction networks (CRNs) taken with mass action kinetics
arise often in {chemistry}, biology and engineering. The dynamics of a mass action system (MAS) can be modeled by polynomial ordinary differential equations (ODEs) with respect to the concentration of each chemical species that is a nonnegative real number. Although the concentration
evolution takes place in a certain invariant set termed {the} stoichiometric
compatibility class, these ODEs will typically lead to a high-dimensional and complex nonlinear system, often referred to as a kinetic system, with unknown model parameters. The model parameters and the network structure both have great effect on the kinetic system, so it is usually extremely difficult to characterize the dynamical properties. However, the past decades witnessed the rapid progress in related directions, including characterizing stability \cite{Al-Radhawi2016, Anderson15, Donnell13, Feinberg87, Horn72, Rao13, Sontag2001, Schaft13}, multi-stationarity \cite{Craciun2006, Feinberg1995, Joshi12} and persistence \cite{ Anderson2011, Angeli2007, Angeli2011,Craciun2013,Pantea2012}. Now, the dynamical properties have gradually been coming to light, especially for those network systems with special structures, such as reversible and weakly reversible structures. In line with these studies, this paper continues to be concerned with characterizing (local asymptotic) stability of equilibria in general MASs.

The pioneering work on stability analysis of equilibria in MASs originated from
\cite{Horn72}, in which complex balanced MASs were proved to have no self-sustaining oscillation or bistability. This paper \cite{Horn72} reported the famous Deficiency Zero Theorem, saying that for a deficiency zero weakly reversible CRN with any choice of rate constants, the dynamical equation {admits a unique positive equilibrium in each positive stoichiometric compatibility class, and these positive equilibria are locally asymptotically stable.} This theorem was further extended to the Deficiency One Theorem \cite{Feinberg1995}, which gives alternative conditions for the existence and uniqueness of equilibrium
in each positive stoichiometric compatibility class, i.e., weakly reversible but not necessary deficiency zero. Recently, the stability issue of complex balanced MASs \cite{Rao13}, and of detailed balanced MASs (a special case of complex balanced systems) \cite{Schaft13} were revisited, and a compact mathematical formulation about the network were derived that allows to characterize the set of positive equilibria and their stability properties, easily. For complex balanced systems, if they are assumed to admit persistence
\cite{Feinberg87}, then the stability of equilibria is global \cite{Anderson2011,Pantea2012}.


In this paper, we consider a class of CRNs, like
\begin{equation*}
3X_2\to 3X_1\to 2X_1+X_2
\end{equation*}
which are not weakly reversible, for characterising stability of the corresponding MASs.
This class of CRNs are often encountered in enzymic catalytic reactions and epidemic models \cite{Hethcote2000}. With respect to stability analysis, the key issue is how to construct the Lyapunov function. Anderson et. al.
\cite{Anderson15}
treated scaling limits of non-equilibrium potential as Lyapunov functions for some special species balanced systems, for example, birth-death systems, etc. Al-Radhawi and Angeli
\cite{Al-Radhawi2016}
developed Lyapunov functions that are piecewise linear in rates. {The existence of such Lyapunov functions could guarantee stability of equilibria}, and further serve to establish asymptotic stability if the Lyapunov functions also satisfy the LaSalle's condition. Clearly, it poses a great challenge to find a Lyapunov function for this class of MASs. Another alternative strategy for capturing their stability properties is to apply the (asymptotic) stability results of complex balanced systems, after all, for which there is a ready-made Lyapunov function, the {pseudo-Helmholtz function}. Linear conjugacy \cite{Johnston2011,Johnston2012,Johnston2013}
is a related notion for this purpose that defines a non-degenerative linear transformation to map the mass action system under study to another one. If the transformed system is complex balanced, then the original system has local asymptotic stability
\cite{Johnston2011}. To obtain expected linear conjugacy, this strategy usually calls for solving linear programming and mixed-integer linear programming problems, analogous to computing complex balanced or detailed balanced realizations \cite{Szederkenyi2011a}.

Motivated by the facts that the {pseudo-Helmholtz function} is a valid Lyapunov function for asymptotic stability analysis for complex balanced MASs and general MASs admitting a complex balanced MAS as the linear conjugacy, we develop the generalized {pseudo-Helmholtz function} aiming to characterize asymptotic stability for general MASs. Based on the mass conservation law, the reconstruction and reverse reconstruction notions are adaptively developed, which are highly related to the original network in properties of equilibria and stability. More importantly, the dynamics of the reverse reconstruction is equivalent to that of the original network at the free species. The stability analysis for the latter is thus transformed into that for the former. Taking the generalized {pseudo-Helmholtz function} as the Lyapunov function, the reverse reconstruction is proved asymptotically stable when the corresponding reconstruction is complex balanced. The key point lies that under this condition the reverse reconstruction possesses only an equilibrium in each positive stoichiometric compatibility class. Utilizing {dynamical} equivalence and mass conservation law, the state of the original network is proved to converges asymptotically to the equilibrium of interest. To facilitate applications of the proposed strategy of stability analysis, we also present a systematic method to compute complex balanced reconstructions based on the computation algorithms for realizations \cite{Szederkenyi2010,Szederkenyi2011a,Szederkenyi2011b}.

The remainder of this paper is organized as follows. Section \ref{sec:2} gives a brief introduction on CRNs and further formulates the research motivation. In Section \ref{sec3}, the generalized {pseudo-Helmholtz function} is developed with the potential to behave as a Lyapunov function. Section \ref{sec4} defines the reconstruction concept to match the use of the generalized {pseudo-Helmholtz function} for characterizing stability of general MASs. This is followed by further defining reverse reconstruction and complex balanced reconstruction to characterize asymptotic stability of MASs in Section \ref{sec:4}. Section \ref{sec:5} provides a systematic method to compute complex balanced reconstructions, and then illustrates it through some representative examples. Finally, Section \ref{sec:6} concludes the paper.

Notations: Throughout the paper, $\mathbb{R}^{n}$, $\mathbb{R}^{n}_{> 0}$, $\mathbb{R}^{n}_{\geq 0}$, $\mathbb{Z}^{n}$, $\mathbb{Z}^{n}_{\geq 0}$ denote the space of $n$ dimensional real, positive real, nonnegative real, integer and nonnegative integer vectors, respectively; for any $x\in \mathbb{R}^{n}_{> 0}$, the mapping
$\mbox{Ln}: \mathbb{R}^{n}_{\textgreater 0}\rightarrow \mathbb{R}^{n}, x\mapsto \mbox{Ln}(x)$ is defined by $(\mbox{Ln}(x))_{i}:=\ln(x_{i})$, and {the mapping $\mbox{Exp}(x)$ is similarly defined by $(\mbox{Exp}(x))_i=\mbox{exp}(x_i)$}; for matrix $A$, $A_{\cdot i}$ refers to the $i$th column of $A$; $\mathbbold{0}_n$ and $\mathbbold{1}_n$ denote the vector with all the entries equal to $0$ and $1$, respectively; matrix $D=\text{diag}(d_{i})\in\mathbb{R}^{n\times n}$ is a diagonal matrix with the $i$th diagonal element to be $d_{i}$; $E_{n}$ denotes the $n\times n$ identity matrix; $\mathrm{Card}(\cdot)$ represents the number of elements in set.

\section{Chemical reaction networks and Motivation statement}\label{sec:2}
In this section, we will review some basic concepts related to CRNs
\cite{Feinberg1995, Horn72} and set forth the motivation of the current study.

\subsection{Chemical reaction networks}
For a CRN involving $n$ species $X_i~(i=1,\cdots,n)$, $c$ complexes $\mathcal{C}_l~(l=1,\cdots,c)$ and $r$
reactions $\mathcal{R}_j~(j=1,\cdots,r)$, it is defined as follows
\cite{Feinberg1995}.

\begin{definition}[Chemical Reaction Network]\label{def:1}
A CRN is composed of three finite sets:
\begin{itemize}
  \item[\rm{(1)}] A set {$\mathcal{X}=\bigcup_{i=1}^{n}\{X_{i}\}$} of chemical species;
  \item[\rm{(2)}] A set {$\mathcal{C}=\bigcup_{j=1}^{r}\{Z_{\cdot j},Z'_{\cdot j}\}$} of
  complexes {with $\mathrm{Card}(\mathcal{C})=c$, $Z_{\cdot j}$, $Z'_{\cdot j}\in\mathbb{Z}^n_{\geq 0}$ and the $i$th entry of $Z_{\cdot j}$ representing the coefficient of $X_i$ in this complex};
  \item[\rm{(3)}] A set {$\mathcal{R}=\bigcup_{j=1}^{r}\{Z_{\cdot j} \to Z'_{\cdot j}\}$} of reactions such that $\forall~Z_{\cdot j}\in\mathcal{C}$, $Z_{\cdot j}\to Z_{\cdot j}\notin\mathcal{R}$ but $\exists~Z'_{\cdot j}\in\mathcal{C}$ to support either $Z_{\cdot j}
  \to Z'_{\cdot j}\in\mathcal{R}$ or $Z'_{\cdot j} \to Z_{\cdot j}\in\mathcal{R}$.
\end{itemize}
The triple $(\mathcal{X,C,R})$ is usually used to express a CRN.
\end{definition}

The complexes $Z_{\cdot j},Z'_{\cdot j}$ actually express the existence
of all of $n$ species in the reactant complex and the product complex, respectively. {If $Z_{ij}\neq 0$ then there includes the species $X_i$ in the complex $Z_{.j}$, and vice versa.}
{Organize distinct complexes according to} $Z=[Z_{\cdot 1},\cdots,Z_{\cdot c}]\in\mathbb{Z}^{n\times c}_{\geq 0}$ to form the complex stoichiometric matrix, whose $l$th column expresses the $l$th complex $Z_{\cdot l}$ in the CRN.

For any CRN, its structure can be described by a digraph, termed as reaction graph
\cite{Horn72}, with the vertices indicating the complexes and the directed edges representing the reactions. Every tail vertex in a directed edge denotes a reactant complex while the head vertex denotes a product complex. Mathematically, this digraph can be defined by an incident matrix $B\in\mathbb{Z}^{c\times r}$ with $B_{lj}=1$ if vertex $l$ is the head vertex of edge $j$ and $B_{lj}=-1$ if vertex $l$ is the tail vertex of edge $j$, while $B_{lj}=0$ otherwise. The incident matrix and the complex stoichiometric matrix further define the
stoichiometric matrix $S\in \mathbb{Z}^{n\times r}$ satisfying
\begin{equation}\label{Stoichiomatrix}
S=ZB.
\end{equation}
It is not difficult to find that the $j$th column of $S$ is given by $S_{\cdot j}=Z'_{\cdot j}-Z_{\cdot j}$, $j=1,\cdots,r$, which is also called the reaction vector of the $j$th reaction. The rank of $S$ usually represents the rank the CRN $(\mathcal{X,C,R})$, denoted by $s$, satisfying $s\leq\min\{n,r\}$.

\begin{definition}[Weakly Reversible CRN]\label{def:2}
A CRN $(\mathcal{X,C,R})$ is weakly reversible if for each
reaction $Z_{\cdot j}\to Z'_{\cdot j}\in\mathcal{R}$, there exists
a sequence of reactions in $\mathcal{R}$, which starts with {$Z'_{\cdot j}$} and ends with
$Z_{\cdot j}$, i.e., {$Z'_{\cdot j} \to Z_{\cdot j_{1}}\in\mathcal{R}$, $Z_{\cdot j_{1}} \to Z_{\cdot j_{2}}\in\mathcal{R}$}, $\cdots$, $Z_{\cdot j_{(m-1)}} \to Z_{\cdot j_{m}}\in\mathcal{R}$, $Z_{\cdot j_{m}} \to Z_{\cdot j}\in\mathcal{R}$.
\end{definition}

\begin{definition}[Stoichiometric Compatibility Class]\label{def:3}
For a CRN $(\mathcal{X,C,R})$, the linear subspace $\mathrm{Im}(S)=\mathrm{span}\{Z'_{\cdot
1}-Z_{\cdot 1},\cdots,Z'_{\cdot r}-Z_{\cdot r}\}$ is called the
stoichiometric subspace. Let $x_0\in\mathbb{R}^n_{\geq 0}$, then the sets $\mathscr{S}(x_0)=\{x_0+\xi|\xi\in\mathrm{Im}(S)\}$, $\mathscr{S}(x_0)\cap\mathbb{R}^n_{\geq 0}$ and $\mathscr{S}(x_0)\cap\mathbb{R}^n_{> 0}$ are named the stoichiometric compatibility class, the nonnegative stoichiometric compatibility class and the positive stoichiometric compatibility class of $x_0$, respectively.
\end{definition}

When a CRN is assigned a mass action kinetics, the rate $v_j$ for reaction $Z_{\cdot j}\to Z'_{\cdot j}$ is evaluated by
\begin{equation}\label{ReactionRate}
v_j(x)=k_jx^{Z_{\cdot j}}:=k_j\prod_{i=1}^{n}
x_{i}^{Z_{ij}},
\end{equation}
where $k_j\in\mathbb{R}_{\textgreater 0}$ is the rate constant for this reaction, and $x\in\mathbb{R}^n_{\geq 0}$ is the vector of concentrations $x_i$ of the chemical species $X_i$.

\begin{definition}[Mass Action System]
{A mass action system is a CRN $(\mathcal{X,C,R})$ taken together with a mass action kinetics.} Denote the vector of reaction rate constants by $\mathcal{K}=(k_1,\cdots,k_r)$ with $k_j$ representing the rate constant for reaction $Z_{\cdot j}\to Z'_{\cdot
j},~j=1,\cdots,r$, then a MAS is often referred to as a quadruple $(\mathcal{X},\mathcal{C},\mathcal{R},\mathcal{K})$.
\end{definition}

The dynamics of a MAS $(\mathcal{X},\mathcal{C},\mathcal{R},\mathcal{K})$ that captures the changes of concentrations of every species is thus expressed as
\begin{equation}\label{CRN-dynamics0}
\dot{x}=Sv(x),\quad\quad x\in\mathbb{R}^n_{\geq 0},
\end{equation}
where $v(x)\in\mathbb{R}_{\geq 0}^{r}$ is the vector of reaction rate $v_j(x)$ of $\mathcal{R}_j$, $j=1,\cdots,r$.

\begin{remark}\label{StateEvolution}
Whatever the initial state $x_0$ is, the state of a MAS $(\mathcal{X},\mathcal{C},\mathcal{R},\mathcal{K})$ will evolve in the nonnegative stoichiometric compatibility class of $x_0$, i.e., in $\mathscr{S}(x_0)\cap\mathbb{R}^n_{\geq 0}$, which can be easily verified by integrating \cref{CRN-dynamics0} from $0$ to $t$, that is,
\begin{equation}\label{statepoint}
  x=x_{0}+\sum_{j=1}^{r}S_{\cdot j}\int_{0}^{t}v_{j}(\tau)d\tau.
\end{equation}
\end{remark}

\begin{remark}
The dynamical equation \cref{CRN-dynamics0} is essentially a polynomial system of ODEs, often referred to as a kinetic system. Note that {there is} not a one-to-one relation between a kinetic system and a MAS. Theoretically, the former can be realized by structurally and/or parametrically different MASs that are dynamically equivalent \cite{Horn72}.
\end{remark}

{Here, the dynamical equivalence is defined as follows.}
\begin{definition}[{Dynamical Equivalence}]\label{def:33}
{Two MASs are said to be dynamically equivalent if they have identical differential equations produced under the assumption of mass action kinetics.}
\end{definition}

\begin{definition}[{Equilibrium}]\label{def:4}
A concentration vector $x^{\ast}\in \mathbb{R}^{n}_{> 0}$ is an equilibrium of a MAS $(\mathcal{X},\mathcal{C},\mathcal{R},\mathcal{K})$ if the dynamics of \cref{CRN-dynamics0} supports $Sv(x^{\ast})=\mathbbold{0}_n$, and a complex balanced equilibrium if $Bv(x^{\ast})=\mathbbold{0}_c$. A MAS $(\mathcal{X},\mathcal{C},\mathcal{R},\mathcal{K})$ is called complex balanced if it admits a complex balanced equilibrium.
\end{definition}

\begin{remark}
Any complex balanced equilibrium is {an equilibrium}, but not necessarily
the other way around, since $Z$ need not be injective.
\end{remark}

{A nice property of complex balanced MASs is that each positive stoichiometric compatibility class contains only one equilibrium, and each such equilibrium is locally asymptotically stable\mbox{\cite{Horn72}}.}

\begin{theorem}
[Locally asymptotic stability of complex balanced equilibrium \cite{Horn72}]\label{thm:1} Let $x^*\in\mathbb{R}^{n}_{> 0}$ be any complex balanced equilibrium of the MAS $(\mathcal{X},\mathcal{C},\mathcal{R},\mathcal{K})$, then $x^*$ is locally asymptotically stable with respect to any initial condition in $\mathscr{S}(x^{\ast})\cap\mathbb{R}^{n}_{> 0}$ near $x^*$.
\end{theorem}

The proof of the above theorem is carried through with the following well-known Lyapunov function
\begin{equation}\label{Gibbs}
G(x)=\sum_{i=1}^{n}(x^*_{i}-x_{i}-x_{i}\text{ln}\frac{x^*_{i}}{x_{i}}), \quad x\in \mathbb{R}^{n}_{> 0}
\end{equation}
which is referred to as {the pseudo-Helmholtz function} \cite{Horn72} in the context.

\begin{remark}
Note that the domain of {the pseudo-Helmholtz function} $G(x)$ is $\mathbb{R}^{n}_{> 0}$, which seems not completely consistent with the nonnegative stoichiometric compatibility class of the equilibrium, i.e., $\mathscr{S}(x^{\ast})\cap\mathbb{R}^{n}_{\geq 0}$, the domain of the complex balanced MAS. {Actually, the domain of the complex balanced MAS can shrink to the positive stoichiometric compatibility class of the equilibrium, $\mathscr{S}(x^{\ast})\cap\mathbb{R}^{n}_{\textgreater 0}$, as long as the initial point is selected carefully\mbox{\cite{Feinberg1995}}. The reason is that $\lim_{x\to\mathbbold{0}_n} G(x)=\sum_{i=1}^n x^*_i$ and $\frac{dG(x)}{dt}\leq 0$, as long as the energy value at an initial point, evaluated by $G(x)$, is lower than $\lim_{x\to\mathbbold{0}_n} G(x)=\sum_{i=1}^n x^*_i$, then the trajectory will not run towards zero point.} A feasible initial points set is $\{x_0|G(x_0)\textless\lim_{x\to\mathbbold{0}_n} G(x)\}\cap\mathscr{S}(x^{\ast})\cap\mathbb{R}^{n}_{\textgreater 0}$.
\end{remark}

\subsection{Motivation statement}
The {pseudo-Helmholtz function} $G(x)$ exhibits high efficacy in characterizing the asymptotic stability of a complex balanced equilibrium, {but it is not in general Lyapunov function for systems that are not complex balanced.} In order to reach the latter purpose, it is necessary to develop new Lyapunov functions. Due to high relevance between complex balanced equilibria and general ones (the {former} is a subset of the latter), a possible solution to find the new Lyapunov function is to extend {the pseudo-Helmholtz function} to a more general form. A naive extension is to multiply every term in the function $G(x)$ by a positive number $d_i,~i=1,\cdots,n$, then we get a new function
\begin{equation}\label{newGibbs}
G_e(x)=\sum_{i=1}^{n}d_i(x^*_{i}-x_{i}-x_{i}\text{ln}\frac{x^*_{i}}{x_{i}}), \quad x\in\mathbb{R}^{n}_{> 0},d_i\in\mathbb{R}_{\textgreater 0}.
\end{equation}
Clearly, like $G(x)$, this extension also has a strict minimum 0 at $x^*$, which indicates that $G_e(x)$ satisfies the necessary condition to behave as a Lyapunov function, and is thus potential for stability analysis for some MASs. Following this superficial motivation, we will discuss a more general extension of {the pseudo-Helmholtz function} so that more MASs may be served, and analyze its properties from viewpoint of becoming the Lyapunov function in the following.

\section{Generalized pseudo-Helmholtz function}\label{sec3}
A more general extension of {the pseudo-Helmholtz function} is to further set the number of terms in $G(x)$ not fixed. This leads to the definition of the generalized {pseudo-Helmholtz function}.

\begin{definition}[Generalized {pseudo-Helmholtz function}]
Let $x^{\ast}\in \mathbb{R}^{n}_{> 0}$ be an equilibrium of a MAS $(\mathcal{X,C,R,K})$ governed by \cref{CRN-dynamics0}. The function $\tilde{G}:\mathbb{R}^n_{\textgreater 0}\mapsto\mathbb{R}_{\geq 0}$ given by
\begin{equation}\label{generalized Gibbs0}
  \tilde{G}(x)=\sum_{i=\tau_1}^{\tau_p} d_{i}(x^*_{i}-x_{i}-x_{i}\ln\frac{x^*_{i}}{x_{i}})
\end{equation}
is called the generalized {pseudo-Helmholtz function}, where $d_i\textgreater 0$, {$\tau_1\leq\cdots\leq\tau_p$} and $\{\tau_1,\cdots,\tau_p\}\subseteq\mathcal{I}=\{1,\cdots,n\}$.
\end{definition}

\begin{remark}\label{remarkGibbs}
For any $x\in\mathbb{R}^n_{\textgreater 0}$, the generalized {pseudo-Helmholtz function} satisfies $\tilde{G}(x)\geq 0$ with equality hold if and only if $(x_{\tau_1},\cdots,x_{\tau_p})^\top=(x_{\tau_1}^*,\cdots,x_{\tau_p}^*)^\top$.
\end{remark}

{For simplicity but without loss of generality, we set $\{\tau_1,\cdots,\tau_p\}=\{1,\cdots,p\}$ in \mbox{(\ref{generalized Gibbs})} for the subsequent analysis}, i.e., keeping the first $p$ terms of $G_e(x)$. As a result,
\begin{equation}\label{generalized Gibbs}
  \tilde{G}(x)=\sum_{i=1}^{p} d_{i}(x^*_{i}-x_{i}-x_{i}\ln\frac{x^*_{i}}{x_{i}}).
\end{equation}
Denote $(x_1,\cdots,x_p)^\top=x_{\bot}$ and $(x_{p+1},\cdots,x_n)^\top=x_{\top}$, then $\forall~x_{\top}\in\mathbb{R}^{n-p}_{\textgreater 0}$ the function $\tilde{G}(x)$ {attains the minimum} $0$ at $x=({x_{\bot}^{\ast}}^\top,x_{\top}^\top)^\top$. This implies that the equilibrium
\begin{equation*}\label{with equality for x}
x^*=\begin{bmatrix}
x_{\bot}^{\ast} \\
x_{\top}^{\ast}
\end{bmatrix}
\end{equation*}
is not the strict minimum point for the generalized {pseudo-Helmholtz function}. At this point, $\tilde{G}(x)$ seems not able to behave as a Lyapunov function. {However, if the degree of freedom of the concerning MAS is $p$, and the remaining $n-p$ state variables may be determined uniquely by those free $p$ variables, then $\tilde{G}(x)$ {may be} a Lyapunov function.} We analyze the number of non-free variables in a MAS $(\mathcal{X,C,R,K})$ through considering the law of mass conservation.

\begin{definition}[Mass conservation law \cite{Horn72}]
{The law of conservation of mass states that mass in an isolated system is neither created nor destroyed by chemical reactions or physical transformations. For a MAS $(\mathcal{X,C,R,K})$ governed by \mbox{\cref{CRN-dynamics0}}, this means that there exists a $n$-dimensional positive vector $\rho\in \mathbb{R}^{n}_{> 0}$ such that $\rho^{\top}S=\mathbbold{0}_{r}^{\top}$, and $\rho$ is called a conserved vector.}
\end{definition}

\begin{remark}\label{ConservedVector}
{Clearly, $\rho\in\mathrm{Ker}(S^{\top})$. Moreover, for a mass-conserved MAS, there must exist a conserved vector, the most simple candidate for which is the vector $M$ of molecular weight $M_i$ of the species $X_i$, $i=1,\cdots,n$.}
\end{remark}

We then define an important concept, conserved matrix, that will play a key role on the subsequent analysis.

\begin{definition}[Conserved Matrix]\label{defConservedMatrix}
{Consider a mass-conserved MAS $(\mathcal{X,C,R,K})$ governed by \mbox{\cref{CRN-dynamics0}} with $\mathrm{rank}(S)=s$. Assume that there exist $q$ linearly independent vectors $\xi_m\in\mathbb{R}^n_{> 0}~(m=1,\cdots,q)$ in $\mathrm{Ker}(S^{\top})$, where $1\leq q\leq\mathrm{dim}\left(\mathrm{Ker}(S^{\top})\right)=n-s$, then the matrix $C\in\mathbb{R}^{n\times q}_{\textgreater 0}$ equipped with $\xi_m$ as column vectors is called a conserved matrix for $(\mathcal{X,C,R,K})$.}
\end{definition}

\begin{remark}
For a mass-conserved MAS, the conserved matrix is not unique, but all of them are full column rank. Under the influence of the same conserved matrix $C$, the state of the MAS in $\mathscr{S}(x^*)\cap\mathbb{R}^{n}_{\geq 0}$ is constrained by this matrix and the equilibrium point $x^*$ according to $C^{\top}x=C^{\top}x^*$. For a MAS not admitting the mass conservation law, the conserved matrix does not exist. In this case, $q=0$.
\end{remark}

Based on the conserved matrix, we can be aware of the number of non-free variables in a mass-conserved MAS.

\begin{theorem}\label{DegreeofFreedom}
For a {mass-conserved} MAS $(\mathcal{X,C,R,K})$ described by \cref{CRN-dynamics0}, $C\in\mathbb{R}^{n\times q}_{\textgreater 0}$ $(q\textgreater 0)$ is a conserved matrix and $x^*$ is an equilibrium point, then in $\mathscr{S}(x^*)\cap\mathbb{R}^{n}_{\geq 0}$ there are {at least} {$q$ non-free state variables, and moreover, these variables are uniquely determined by other $n-q$ state variables and $x^*$ but independent of the selection of $C$.}
\end{theorem}

\begin{proof}
{For this mass-conserved MAS, $\forall x\in\mathscr{S}(x^*)\cap\mathbb{R}^{n}_{\geq 0}$, we have $C^{\top}x=C^{\top}x^*$. Note that rank$(C)=q$ ($1\leq q\leq\mathrm{dim}\left(\mathrm{Ker}(S^{\top})\right)$), so there exist {at least} $q$ non-free state variables among all $n$ variables contained in the MAS. For simplicity but without loss of generality, the last $q$ state variables from $x_{n-q+1}$ to $x_{n}$ are thought as {non-free variables}, then we get}
\begin{equation}\label{mass conservtion relation 1}
  C^{\top}x=\begin{bmatrix}
  C_{l}^{\top} & C_{r}^{\top}
  \end{bmatrix}\begin{bmatrix}
  x_{\bot} \\
  x_{\top}
  \end{bmatrix}
  =C_{l}^{\top}x_{\bot}+C_{r}^{\top}x_{\top}
  =C^{\top}x^*,
\end{equation}
where $C_{l}^{\top}\in \mathbb{R}^{q\times (n-q)}_{> 0}$, $C_{r}^{\top}\in \mathbb{R}^{q\times q}_{> 0}$ are two sub-blocks of $C^\top$, $x_\bot=(x_1,\cdots,x_{n-q})^\top$ and $x_\top=(x_{n-q+1},\cdots,x_n)^\top$. Since $C$ is full column rank, $C_{r}^{\top}$ is invertible. {Also, utilizing $C^{\top}x^*=C_{l}^{\top}x_{\bot}^*+C_{r}^{\top}x_{\top}^*$}, the last equality of \cref{mass conservtion relation 1} may be rewritten as
\begin{equation}\label{Dependence}
x_{\top}=C_{r}^{-\top}C^{\top}_l(x^*_{\bot}-x_{\bot})+x_{\top}^*,
\end{equation}
which indicates that $x_\top$ is uniquely determined by $C$, $x^*$ and $x_\bot$.

Since the conserved matrix $C$ is not unique for the MAS, we need to further prove the above equality independent of $C$. {Assume that $C'\in\mathbb{R}^{n\times q}_{\textgreater 0}$ is another conserved matrix for the MAS under consideration, which has the same dimension as $C$. Likewise, we utilize the conservation relation ${C'}^\top x={C'}^\top x^*$, and further partition $C'$ according to the same pattern as $C$, then we obtain} $${x_{\top}={C'}_{r}^{-\top}{C'}_{l}^{\top}(x_{\bot}^*-x_{\bot})+x_{\top}^*.}$$ {By combining it with \mbox{\cref{Dependence}}, we get $C_{r}^{-\top}C^{\top}_l(x_{\bot}^*-x_{\bot})={C'}_{r}^{-\top}{C'}_{l}^{\top}(x_{\bot}^*-x_{\bot}).$ This means that whatever the conserved matrix is, $C_{r}^{-\top}C^{\top}_l(x_{\bot}^*-x_{\bot})$ is always equal to ${C'}_{r}^{-\top}{C'}_{l}^{\top}(x_{\bot}^*-x_{\bot})$.} Therefore, $x_\top$ is uniquely determined by $x^*$ and $x_{\bot}$ but independent of the selection of $C$, which completes the proof.
\end{proof}

Based on this theorem, we can thus address the issue that the generalized {pseudo-Helmholtz function} has a strict minimum $0$ at $x^*$ through setting $p=n-q$ in \cref{generalized Gibbs}.

\begin{theorem}\label{theoremGibbs}
For a {mass-conserved} MAS $(\mathcal{X,C,R,K})$ captured by \cref{CRN-dynamics0}, let $x^*$ be an equilibrium,  and $C\in\mathbb{R}^{n\times q}_{\textgreater 0}$ be a conserved matrix. Then the generalized {pseudo-Helmholtz function} given by
\begin{equation}\label{generalized Gibbs1gao}
\tilde{G}(x)=\sum_{i=1}^{n-q} d_{i}(x^*_{i}-x_{i}-x_{i}\ln\frac{x^*_{i}}{x_{i}}),\quad {\forall x\in\mathbb{R}^n_{\textgreater 0}}~{and~\forall d_i\textgreater 0,}
\end{equation}
has a strict minimum at $x^*$ {within $\mathscr{S}(x^*)\cap\mathbb{R}^n_{\textgreater 0}$}, while
\begin{equation}\label{partialGibbs}
\nabla\tilde{G}(x)=\mathrm{diag}(d_1,\cdots,d_{n-q},0,\cdots,0)\nabla G(x).
\end{equation}
\end{theorem}

\begin{proof}
From \cref{remarkGibbs}, we have $\tilde{G}(x)\geq 0$ with equality hold if $x_\bot=x_\bot^*\in\mathbb{R}^{n-q}_{\textgreater 0}$ and $x_\top\in\mathbb{R}^q_{\textgreater 0}$ may be arbitrary in its domain. Also from \cref{DegreeofFreedom}, we get that $x_\top$ is uniquely determined by $x_\bot^*$ in $\mathscr{S}(x^*)\cap\mathbb{R}^n_{\textgreater 0}$, i.e., $x_\top=x_\top^*$. This means that $\tilde{G}(x)$ has a strict minimum at $x^*$ in $\mathscr{S}(x^*)\cap\mathbb{R}^n_{\textgreater 0}$. It is straightforward to obtain \cref{partialGibbs} if $\tilde{G}(x)$ and $G(x)$ in \cref{Gibbs} are both taken partial derivative with respect to $x$.
\end{proof}

\begin{remark}
The more the linearly independent conserved vectors are for a {mass-conserved} MAS, the more sparse the diagonal matrix $\mathrm{diag}(d_1,\cdots,d_{n-q},0,\cdots,0)$ is, which implies that the dynamical equation \cref{CRN-dynamics0} of this MAS is more reduced. However, it is extremely difficult for a MAS to find a maximum linear independence group of all conserved vectors, i.e., maximizing $q$. {In fact, this is not necessary, since a portion of linear independence group, such as a conserved vector composed of the molecular weight of all species, is enough to work as a conserved matrix to derive nice properties of the generalized pseudo-Helmholtz function, like strict convexity at equilibrium.}
\end{remark}

\begin{remark}
If the MAS does not support the mass conservation law, the generalized {pseudo-Helmholtz function} reduce to the naive extension of the {pseudo-Helmholtz function}, i.e., $G_e(x)$ in \cref{newGibbs} which, obviously, can behave as a Lyapunov function.
\end{remark}

The generalized pseudo-Helmholtz function, {as discussed in \mbox{\cref{remarkGibbs}}, is nonnegative} and has lower bound only at the equilibrium, so it may be a Lyapunov function. However, it is still needed the time derivative of $\tilde{G}(x)$ to follow $\dot{\tilde{G}}(x)\leq 0$, i.e., we should have
\begin{equation}\label{GeneralizedGibbsDerivative}
\begin{aligned}
\dot{\tilde{G}}(x)&=\nabla^\top\tilde{G}(x)Sv(x) \\
&=\nabla^\top G(x)\mathrm{diag}(d_1,\cdots,d_{n-q},0,\cdots,0)Sv(x) \\
&\leq 0.
\end{aligned}
\end{equation}
where the second equality is based on \cref{partialGibbs}. As one might know, $G(x)$ is {an} available Lyapunov function for complex balanced MASs, and its time derivative follows
$\dot{G}(x)=\nabla^\top G(x)Sv(x)$. However, in the result of \cref{GeneralizedGibbsDerivative} there exists a diagonal matrix between $\nabla^\top G(x)$ and $Sv(x)$. In order to fully utilize the performance of $G(x)$ as a Lyapunov function, we manage to eliminate the diagonal matrix $\mathrm{diag}(d_1,\cdots,d_{n-q},0,\cdots,0)$ through defining reconstruction notation in the next section.

\section{Reconstruction}\label{sec4}
In this section, the reconstruction notation will be defined for characterizing stability of equilibria in some MASs.

\subsection{Definition}
The reconstruction concept has some relevance to the realization notation proposed in \cite{Szederkenyi2011a}. We firstly revisit the realization definition.

\begin{definition}[Realization]\label{def:6}
Given a nonnegative autonomous polynomial system $\dot{x}=f(x)$ with $f:\mathbb{R}^n_{> 0}\rightarrow\mathbb{R}^n$ to be locally Lipschitz, if there
exists a MAS $(\mathcal{X,C,R,K})$ modeled by $\dot{x}=Sv(x)$ such that
$f(x)=Sv(x)$, then the autonomous polynomial system is kinetically realizable, and
the MAS $(\mathcal{X,C,R,K})$ is termed as a kinetic realization, or
realization for short, of this polynomial system.
\end{definition}

%

We then give the definition of reconstruction following that of realization.

\begin{definition}\label{def:7}
For a {mass-conserved} MAS $(\mathcal{X,C,R,K})$ described by $\dot{x}=Sv(x)$, {let $x^*\in\mathbb{R}^n_{\textgreater 0}$ be an equilibrium and} $C\in\mathbb{R}^{n\times q}_{\textgreater 0}$ be a conserved matrix. If there exists a positive diagonal matrix $D_1=\mathrm{diag}(d_1,\cdots,d_{n-q})$ and a MAS $(\mathcal{\hat{X},\hat{C},\hat{R},\hat{K}})$ given by $\dot{\hat{x}}=\hat{S}\hat{v}(\hat{x})$ (the parameters follow the same symbols as for $(\mathcal{X,C,R,K})$ but with a hat above) such that
\begin{equation}\label{ReconstructionDefinition1}
{DSv(x)} =\begin{bmatrix}
{\hat{S}\hat{v}(\hat{x})} \\
{\mathbbold{0}_{q}}
\end{bmatrix},
\end{equation}
where
\begin{equation}\label{reconstruction matrix}
  D=\begin{bmatrix}
  D_{1} & \mathbbold{0}_{(n-q)\times q}\\
  C_{l}^{\top} & C_{r}^{\top}
  \end{bmatrix},~{\hat{x}=x_{\bot}\in\mathbb{R}^{n-q}_{\geq 0}},~\hat{S}\in\mathbb{Z}^{(n-q)\times\hat{r}} ~\mathrm{and} ~\hat{v}:\mathbb{R}^{n-q}_{\geq 0}\to\mathbb{R}^{\hat{r}}_{\geq 0},
\end{equation}
then $(\mathcal{\hat{X},\hat{C},\hat{R},\hat{K}})$ is called a reconstruction of $(\mathcal{X,C,R,K})$, and $D$ is termed as the reconstructing matrix from $(\mathcal{X,C,R,K})$ to $(\mathcal{\hat{X},\hat{C},\hat{R},\hat{K}})$.
\end{definition}

{Note that the state variables $\hat{x}$ in the reconstructed system is actually the projection of $x$ onto the first $n-q$ coordinations, and \mbox{\cref{ReconstructionDefinition1}} holds for all $x$ in $\mathscr{S}(x^*)\cap\mathbb{R}^n_{\geq 0}$. Some examples of reconstructions are given in the subsequent \mbox{\cref{table:1}}.}


\begin{remark}\label{rem:3}
The relation \cref{ReconstructionDefinition1} can be equivalently expressed as
\begin{equation}\label{ReconstructionDefinition2}
D_1S_1v(x)=\hat{S}\hat{v}(\hat{x}),
\end{equation}
where $S_1\in\mathbb{Z}^{(n-q)\times r}$ comes from $S=(S_1^\top,S_2^\top)^\top$. In addition, since $D_1$ and $C_r^\top$ are both invertible, the reconstructing matrix $D$ is invertible with the inverse matrix to be
\begin{equation}\label{inverse matrix of reconstruction matrix}
  D^{-1}=\begin{bmatrix}
  D_{1}^{-1} & \mathbbold{0}_{(n-q)\times q}\\
  -C_{r}^{-\top}C_{l}^{\top}D_{1}^{-1} & C_{r}^{-\top}
  \end{bmatrix}.
\end{equation}
The nonsingular property of $D$ is quite important that will be used to eliminate the diagonal matrix $\mathrm{diag}(d_1,\cdots,d_{n-q},0,\cdots,0)$ in the result of $\dot{\tilde{G}}(x)$ in \cref{GeneralizedGibbsDerivative}.
\end{remark}

\begin{remark}\label{LinearConjugacyComparison}
If the MAS $(\mathcal{X,C,R,K})$ dose not support the mass conservation law, i.e., $q=0$, then $D=D_1=\mathrm{diag}(d_1,\cdots,d_n)$. In this case, the reconstruction means to make a positive diagonal transformation from the right hand side of the differential equation $\dot{x}=Sv(x)$ to that of $\dot{\hat{x}}=\hat{S}\hat{v}(\hat{x})$. At this time, the reconstruction is the same with the well-known linear conjugacy concept \cite{Johnston2011,Johnston2012,Johnston2013} which defines a linear transformation $\hat{x}=Dx$, and then by substituting $x=D^{-1}\hat{x}$ into $\dot{x}=Sv(x)$ and dynamical equivalence to $\dot{\hat{x}}=\hat{S}\hat{v}(\hat{x})$ produces the relation $DSv(D^{-1}\hat{x})=\hat{S}\hat{v}(\hat{x})$. Further, if $\forall i,~d_i=1$, then the reconstruction will have the same effect as the realization.
\end{remark}

From \cref{def:7}, apparently, for a MAS even though under the same reconstructing matrix, there may exist a few parametrically and/or structurally different reconstructions that are dynamically equivalent. This property is the same as in the realization concept. However, unlike realization the notion of reconstruction only characterizes a binary relation between MASs, but not between the MAS and a general polynomial system. Therefore, there are at least two MASs involved when we mention reconstruction. The whole process seems to reconstruct a new MAS from a known one. The reconstruction notation together with the generalized {pseudo-Helmholtz function} will play an important role on capturing the stability of some MASs.

\subsection{Reconstruction for characterizing stability}
In this subsetion, we will characterize stability of equilibria in MASs utilizing the reconstruction strategy.

\begin{lemma}\label{LemmaEquilibrium}
Assume that $(\mathcal{\hat{X},\hat{C},\hat{R},\hat{K}})$ is a reconstruction of a {mass-conserved} MAS $(\mathcal{X,C,R,K})$ under the reconstructing matrix $D$ given by \cref{reconstruction matrix}, and their dynamical equations are $\dot{\hat{x}}=\hat{S}\hat{v}(\hat{x})$ and \cref{CRN-dynamics0}, respectively. {If $x^*\in\mathbb{R}^n_{> 0}$ is an equilibrium in $(\mathcal{X,C,R,K})$, then $\hat{x}^*=x_{\bot}^*\in\mathbb{R}^{n-q}_{> 0}$ is an equilibrium in $(\mathcal{\hat{X},\hat{C},\hat{R},\hat{K}})$.}
\end{lemma}

\begin{proof}
{Since $x^*\in\mathbb{R}^n_{> 0}$ is an equilibrium in $(\mathcal{X,C,R,K})$, we get $Sv(x^*)=\mathbbold{0}_n$, i.e., $D_1S_1v(x^*)=\mathbbold{0}_{n-q}$.  From \mbox{\cref{ReconstructionDefinition2}}, we have $D_1S_1v(x^*)=\hat{S}\hat{v}(x_{\bot}^*)$. Therefore, $\hat{x}^*=x_{\bot}^*$ is an equilibrium in $(\mathcal{\hat{X},\hat{C},\hat{R},\hat{K}})$.}
\end{proof}

\begin{theorem}\label{thm:3}
For a {mass-conserved} MAS $(\mathcal{X,C,R,K})$ modeled by \cref{CRN-dynamics0} with an equilibrium $x^*\in\mathbb{R}^n_{> 0}$, assume that the MAS $(\mathcal{\hat{X},\hat{C},\hat{R},\hat{K}})$  in the form of $\dot{\hat{x}}=\hat{S}\hat{v}(\hat{x})$ is its reconstruction, bridged by the reconstructing matrix $D$ defined in \cref{reconstruction matrix}, then $\hat{x}^*\in\mathbb{R}^{n-q}_{> 0}$ is an equilibrium in $(\mathcal{\hat{X},\hat{C},\hat{R},\hat{K}})$. If $\hat{x}^*$ is stable rendered by a Lyapunov function $\hat{V}:\mathbb{R}_{\geq 0}^{n-q}\to\mathbb{R}_{\geq 0}$ with $\hat{V}\in\mathscr{C}^2$ and its Hessian matrix to be a $(n-q)$-dimensional positive diagonal matrix, then $x^*$ is stable in $(\mathcal{X,C,R,K})$.
\end{theorem}

\begin{proof}
According to \cref{LemmaEquilibrium}, $\hat{x}^*$ is naturally an equilibrium in $(\mathcal{\hat{X},\hat{C},\hat{R},\hat{K}})$.
As $\hat{V}\in\mathscr{C}^2$ is a Lyapunov function to render the system $(\mathcal{\hat{X},\hat{C},\hat{R},\hat{K}})$ stable at $\hat{x}^*$, we have $ \hat{V}(\hat{x})\geq 0$ with equality holding if and only if $\hat{x}=\hat{x}^*$, and moreover, $\nabla \hat{V}(\hat{x}^*)=\mathbbold{0}_{n-q}$ and {$\forall~\hat{y}\in\mathbb{R}^{n-q}_{\textgreater 0}$, if $\hat{y}\neq \hat{x}^*$, then $\nabla \hat{V}(\hat{y}) \neq \mathbbold{0}_{n-q}$ and $\partial \hat{V}(\hat{y}) / \partial \hat{x}_i$ is strictly increasing}. In addition,
\begin{equation}\label{stability of mass action system for theorem 3}
\begin{aligned}
\frac{d\hat{V}(\hat{x})}{dt}&=\nabla ^\top\hat{V}(\hat{x})\hat{S}\hat{v}(\hat{x})=\nabla ^\top\hat{V}(\hat{x})D_{1}S_{1}v(x)\leq 0,
\end{aligned}
\end{equation}
where $D_1$ and $S_1$ are given according to \cref{ReconstructionDefinition2}. Construct the following $\mathscr{C}^2$ function $V:\mathbb{R}^n_{\geq 0}\to\mathbb{R}_{\geq 0}$ as
\begin{equation}\label{LyaNrecon}
{V(x)=\int_{x^*}^x\nabla^\top\hat{V}(\hat{y})D_{2} dy,}
\end{equation}
{where $D_{2}=[D_{1},\mathbbold{0}_{(n-q)\times q}]$ and {$y\in\mathbb{R}^n_{\geq 0}$ with $\hat{y}$ to be the projection of $y$ onto the first $n-q$ coordinates}. Note that $x^*$ and $x$ are vectors, so the integral is over the line connecting these two points. }

It is clear that $V(x^*)=0$ and $\nabla V(x^*)=D_2^\top\nabla \hat{V}(\hat{x}^*)=\mathbbold{0}_n$.
Utilizing the mean value theorem of integrals, we can rewrite \cref{LyaNrecon} as
\begin{equation*}\label{LyaNrecon1}
{V(x)=\nabla^\top\hat{V}(\hat{\bar{x}})D_{2}(x-x^*)=\sum_{i=1}^{n-q}d_i\frac{\partial \hat{V}(\hat{\bar{x}})}{\partial \hat{x}_i}(x_i-x_i^*),}
\end{equation*}
where $\bar{x}\in\mathscr{S}(x^*)\cap\mathbb{R}^n_{\textgreater 0}$ lies between $x^*$ and $x$, {and $\hat{\bar{x}}$ is the projection of $\bar{x}$ onto the first $n-q$ coordinates. Since $\frac{\partial \hat{V}(\hat{\bar{x}})}{\partial \hat{x}_i}$ is strictly increasing and $\nabla \hat{V}(\hat{x}^*)=\mathbbold{0}_{n-q}$, we have $\frac{\partial \hat{V}(\hat{\bar{x}})}{\partial \hat{x}_i}(x_i-x_i^*)\geq 0$.} Therefore, $V(x)=0$ if and only if $x_\bot=x_\bot^*$, but $x_\top$ may be arbitrary, where $x_\bot=(x_1,\cdots,x_{n-q})^\top$ and $x_\top=(x_{n-q+1},\cdots,x_{n})^\top$. Also since $x_\top$ is uniquely determined by $x_\bot^*$ in $\mathscr{S}(x^*)\cap\mathbb{R}^n_{\textgreater 0}$, we have $V(x)=0$ iff $x=x^*$. Again from \cref{LyaNrecon} we can write the Hessian matrix of $V(x)$ to be
\begin{equation*}
H(V(x))=\begin{bmatrix}
  \mathrm{diag} \left(d_{i}\frac{\partial^{2}\hat{V}(\hat{x})}{\partial \hat{x}_{i}^{2}}\right) & \mathbbold{0}_{(n-q)\times q}\\
  \mathbbold{0}_{q \times (n-q)} & \mathbbold{0}_{q \times q}
  \end{bmatrix}, \quad i=1,\cdots,n-q.
\end{equation*}
Since the Hessian matrix of $\hat{V}(\hat{x})$ is a positive diagonal matrix, $\forall~x\in\mathscr{S}(x^*)\cap\mathbb{R}^n_{\textgreater 0}$ the Hessian matrix $H(V(x))$ is semi-positive definite. This means that $V(x)$ is a convex function in $\mathscr{S}(x^*)\cap\mathbb{R}^n_{\textgreater 0}$, and $x^*$ is a minimum point for $V(x)$, i.e., $V(x)\geq V(x^*)$. Therefore, we have $V(x)\geq 0$ with equality hold iff $x=x^*$. Further, we have
\begin{equation*}
\begin{aligned}
\frac{dV(x)}{dt}&=\nabla ^\top V(x)Sv(x) \\
&=\nabla
^\top\hat{V}(\hat{x})D_{2}Sv(x) \\
&=\nabla ^{\top}\hat{V}(\hat{x})
\begin{bmatrix}
D_1 & \mathbbold{0}_{(n-q) \times q}
\end{bmatrix}
\begin{bmatrix}
S_1\\S_2
\end{bmatrix}
v(x)  \\
&=\nabla ^\top\hat{V}(\hat{x})D_{1}S_{1}v(x) \\
&\leq 0.
\end{aligned}
\end{equation*}
{where the last inequality holds from \mbox{\cref{stability of mass action system for theorem 3}}}. Moreover, when $x=x^*$, we get $\frac{dV(x)}{dt}=0$. Hence, $x^*$ is a stable equilibrium in the original MAS $(\mathcal{X,C,R,K})$ with $V(x)$ as the Lyapunov function.
\end{proof}

\begin{corollary}\label{cor:1}
{Consider a mass-conserved MAS $(\mathcal{X,C,R,K})$ described by \mbox{\cref{CRN-dynamics0}} with $x^*\in\mathbb{R}^n_{> 0}$ as an equilibrium. Let $(\mathcal{\hat{X},\hat{C},\hat{R},\hat{K}})$ be a reconstruction.} If $\hat{x}^*\in\mathbb{R}^{n-q}_{> 0}$ is a complex balanced equilibrium in the reconstructed system , then $x^*$ is stable in the original MAS.
\end{corollary}

\begin{proof} From \cref{thm:1}, the complex balanced system $(\mathcal{\hat{X},\hat{C},\hat{R},\hat{K}})$ is locally asymptotically stable at $\hat{x}^*$ with $G(\hat{x})$ defined like \cref{Gibbs} as the Lyapunov function. Note that
$G\in\mathscr{C}^2$ and $\frac{\partial G^2(\hat{x})}{\partial
\hat{x}^2}=\mathrm{diag}(1/\hat{x}_1,\cdots,1/\hat{x}_{n-q})$, which satisfy the conditions requested by \cref{thm:3}, therefore, $(\mathcal{X,C,R,K})$ is stable at $x^*$.
\end{proof}

In the following special case, the stability properties can be transferred mutually between the original network and the reconstruction.

\begin{proposition}\label{cor:2}
Consider a MAS $(\mathcal{X,C,R,K})$ with the linear kinetics $\dot{x}=Px$, where $P\in\mathbb{R}^{n\times n}$ is a Metzler matrix. Assume that $x^*\in\mathbb{R}^n_{> 0}$ is an equilibrium in $(\mathcal{X,C,R,K})$ and $(\mathcal{\hat{X},\hat{C},\hat{R},\hat{K}})$ is a reconstruction of $(\mathcal{X,C,R,K})$ under a positive diagonal matrix $D$, then $(\mathcal{X,C,R,K})$ is stable at $x^*$ if and only if $(\mathcal{\hat{X},\hat{C},\hat{R},\hat{K}})$ is stable at $\hat{x}^*\in\mathbb{R}^n_{> 0}$.
\end{proposition}

\begin{proof}
Clearly, $\hat{x}^*$ is an equilibrium in $(\mathcal{\hat{X},\hat{C},\hat{R},\hat{K}})$. Note that the reconstruction $(\mathcal{\hat{X},\hat{C},\hat{R},\hat{K}})$ induces a linear kinetic dynamics $\dot{\hat{x}}=DP\hat{x}$, in which $DP$ is also a Metzler matrix. Based on the $D$-stability of matrices
\cite{Mason2009}, the system $(\mathcal{X,C,R,K})$ is stable at $x^*$ if and only if $(\mathcal{\hat{X},\hat{C},\hat{R},\hat{K}})$ is stable at $\hat{x}^*$.
\end{proof}

As can be seen from \cref{cor:1} and \cref{cor:2}, the stability properties are retained from the reconstruction to the original network usually requesting the former to have special network structure. As far as the current technique is concerned, it is difficult to say that a stable reconstruction with a general structure must lead to a stable original network. We place this issue as a possible point of future research. The high validity of complex balancing in the reconstruction for stability retainment motivates us to further use this strategy to characterize asymptotic stability of the original network.

\section{Asymptotic stability characterized through complex balanced reconstruction}\label{sec:4}
In \cref{thm:3}, we are able to characterize stability of equilibrium points in MASs, but fail to capture {local asymptotic stability}. Here, the stability of an equilibrium means that the system trajectory runs in a bounded ball centered around this equilibrium while {the local asymptotic stability means that the system trajectory will asymptotically converge to this equilibrium if the trajectory starts from a neighborhood of the equilibrium.} The main reason is that there may exist multiple equilibria in the MAS so that the time derivative of the Lyapunov function is negative semidefinite with respect to the concerning equilibrium. In this section, we will address this issue to characterize the local asymptotic stability of some MASs through setting the reconstruction to have a complex balancing structure. We firstly introduce a special reconstruction.
\subsection{Reverse reconstruction}
Reverse reconstruction is a generalized MAS \cite{Horn72}, obtained from the reconstruction for a given MAS, which is defined as follows.

\begin{definition}\label{def:8}
Given a {mass-conserved} MAS $(\mathcal{X,C,R,K})$, let $(\mathcal{\hat{X},\hat{C},\hat{R},\hat{K}})$ denote its reconstruction under the reconstructing matrix $D$ defined in \cref{reconstruction matrix}. A MAS $(\mathcal{\tilde{X},\tilde{C},\tilde{R},\tilde{K}})$ is the reverse reconstruction of $(\mathcal{X,C,R,K})$ with respect to $(\mathcal{\hat{X},\hat{C},\hat{R},\hat{K}})$ if its
complexes set {$\tilde{\mathcal{C}}=\bigcup_{j=1}^{\tilde{r}}
\{\tilde{Z}_{\cdot j},\tilde{Z}'_{\cdot j}\}$} and reactions set
$\tilde{\mathcal{R}}=\bigcup_{j=1}^{\tilde{r}}\{\xymatrix{\tilde{Z}_{\cdot
j} \ar ^{\tilde{k}_{j}~} [r] & \tilde{Z}'_{\cdot j}}\}$ {satisfy}
\begin{itemize}
  \item[\rm{(}\romannumeral1)] $\tilde{r}=\hat{r}$;
  \item[\rm{(}\romannumeral2)] $\tilde{Z}_{\cdot j}=\hat{Z}_{\cdot j},~\tilde{Z}'_{\cdot j}=\hat{Z}_{\cdot j}+D_1^{-1}(\hat{Z}'_{\cdot j}-\hat{Z}_{\cdot j}),~\tilde{k}_j=\hat{k}_j,~\forall~j=1,\cdots,\tilde{r}$.
\end{itemize}
\end{definition}

\begin{remark}\label{gaotong2}
The dynamics of the reverse reconstruction $(\mathcal{\tilde{X},\tilde{C},\tilde{R},\tilde{K}})$ follows \begin{equation}
\dot{\tilde{x}}=\tilde{S}\tilde{v}(\tilde{x})
\end{equation}
where $\tilde{S}\in\mathbb{R}^{(n-q)\times \tilde{r}}$, $\tilde{v}(\cdot)$ is a $\tilde{r}$-dimensional vector valued function, {and $\tilde{x}$ is the projection of $\hat{x}$ onto itself}. {It is clear from the reverse reconstruction definition that $\tilde{S}=D_1^{-1}\hat{S}$ and $\tilde{v}(\tilde{x})=\hat{v}(\hat{x})$.}
\end{remark}

\begin{remark}\label{rem:6}
The underlying CRN in the reverse reconstruction does not fit into the CRN class defined in \cref{def:1}, since the elements in $\tilde{Z}'_{\cdot j}$ are not necessarily nonnegative integer. However, this change does not invalidate the subsequent results. {The main reason is that the reverse reconstruction belongs to generalized MASs, but some results of CRN theory (used for the subsequent analysis)\mbox{\cite{Feinberg87}}, like power-law rate functions, definitions of stoichiometric subspace and stoichiometric compatibility classes, carry over to the case of generalized mass action kinetics\mbox{\cite{Horn72,Minke2017}}. }
\end{remark}

\begin{proposition}\label{propo:4}
For a {mass-conserved} MAS $(\mathcal{X,C,R,K})$ with dynamics \cref{CRN-dynamics0} and an equilibrium $x^*\in\mathbb{R}^n_{\textgreater 0}$, let $(\mathcal{\hat{X},\hat{C},\hat{R},\hat{K}})$ be its reconstruction bridged by the reconstructing matrix $D$ defined in \cref{reconstruction matrix}, and $(\mathcal{\tilde{X},\tilde{C},\tilde{R},\tilde{K}})$ be its reverse reconstruction with respect to $(\mathcal{\hat{X},\hat{C},\hat{R},\hat{K}})$. {Then $\tilde{x}^*=x^*_{\bot}\in\mathbb{R}^{n-q}_{\textgreater 0}$ is an equilibrium in $(\mathcal{\tilde{X},\tilde{C},\tilde{R},\tilde{K}})$, and the dynamics of $(\mathcal{\tilde{X},\tilde{C},\tilde{R},\tilde{K}})$ is equivalent to that of $(\mathcal{X,C,R,K})$ at the first $n-q$ species.}
\end{proposition}

\begin{proof}
{Remark 12 reveals that $\tilde{S}=D_1^{-1}\hat{S}$ and $\tilde{v}(\tilde{x})=\hat{v}(\hat{x})$. Hence, $D_1^{-1}\hat{S}\hat{v}(\hat{x})=\tilde{S}\tilde{v}(\tilde{x})$. Further, utilizing \mbox{\cref{ReconstructionDefinition1}} we get}
\begin{equation*}
{Sv(x)=}
\begin{bmatrix}
{E_{n-q}}\\{-C_r^{-\top}C_l^\top}
\end{bmatrix}
{\tilde{S}\tilde{v}(\tilde{x})},
\end{equation*}
{where $E_{n-q}$ is a $(n-q)$-dimensional unit matrix. Therefore, $S_1v(x)=\tilde{S}\tilde{v}(\tilde{x})$. Note that the left term measures the dynamics of $(\mathcal{X,C,R,K})$ at the first $n-q$ species, i.e., $\dot{x}_\bot(t)=S_1v(x)$. We thus obtain $\dot{x}_\bot(t)=\dot{\tilde{x}}(t)$. $\tilde{S}\tilde{v}(\tilde{x}^*)=\mathbbold{0}_{n-q}$ is straightforward since $Sv(x^*)=\mathbbold{0}_n$. Therefore, the results are true.}
\end{proof}

\begin{corollary}\label{cor:3}
{Assume that $(\mathcal{X,C,R,K})$, $(\mathcal{\hat{X},\hat{C},\hat{R},\hat{K}})$,
$(\mathcal{\hat{X},\tilde{C},\tilde{R},\tilde{K}})$ and $D$ have the same meanings with those in
\mbox{\cref{propo:4}}. If $x^*\in\mathbb{R}^n_{\textgreater 0}$ is an equilibrium in $(\mathcal{X,C,R,K})$, then the reconstruction and reverse reconstruction systems have the common equilibrium point $\hat{x}^*=\tilde{x}^*=x_{\bot}^*$, and moreover, $\mathrm{Im}(\tilde{S})=\mathrm{Im}(D_1^{-1}\hat{S})$.}
\end{corollary}

Reverse reconstruction is a very special MAS, which has the same dynamics with that of the original MAS {at the first $n-q$ species}. Therefore, if an equilibrium point in the original network system is locally
asymptotically stable {whose first $n-q$ entries are naturally locally asymptotically stable}, then this stability behavior can be potentially characterized based on an similar analysis in the reverse reconstruction and on how stability properties are retained under the adopted nonsingular transformation. We tackle these tasks in the following by using a particular network class for reconstruction, i.e., so-called complex balanced MASs, for which the local asymptotic stability has been proved in \cref{thm:1}.

\subsection{Complex balanced reconstruction}
As said in \cref{cor:1}, when a mass-conserved MAS has a complex balanced MAS as a reconstruction, then its equilibrium must be stable. We thus set the reconstruction to be a complex balanced system for further analysis. This kind of reconstructions is referred to as the complex balanced reconstruction in the context. For simplicity of symbols, we denote the set of MASs admitting complex balanced reconstructions by $\mathscr{N}_\mathrm{Com}$, i.e., for any mass-conserved MAS $(\mathcal{X,C,R,K})\in\mathscr{N}_\mathrm{Com}$, there exist a complex balanced MAS, denoted by $(\mathcal{\hat{X}},\hat{\mathcal{C}}_C,\hat{\mathcal{R}}_C,\hat{\mathcal{K}}_C)$, and a reconstructing matrix $D$ defined in \cref{reconstruction matrix} satisfying
\begin{equation}\label{ComplexReconstruction}
 {DSv(x)=}\begin{bmatrix}
{\hat{S}_C\hat{v}_C(\hat{x})} \\
\mathbbold{0}_{q}
\end{bmatrix},
\end{equation}
where $\hat{S}_{C}\in\mathbb{Z}^{(n-q)\times \hat{r}_C}$.


In the case of existence of multiple equilibria in a MAS, the key to extend stability of equilibrium to asymptotic stability is to identify that {there is} only/at most an equilibrium in each positive stoichiometric compatibility class. The main reason is that the state of the MAS will evolve only in a stoichiometric compatibility class, as given in \cref{StateEvolution}. We try to reach this conclusion for the reverse reconstruction firstly, since its dynamics links that of the original MAS in $\mathscr{N}_\mathrm{Com}$ by a constant matrix.

As one may know, there exists a close relation between equilibria for a complex balanced MAS, shown as follows.

\begin{lemma}
[\cite{Feinberg1995,Horn72}]\label{lem:1}
For any complex balanced MAS $(\mathcal{\hat{X}},\hat{\mathcal{C}}_C,\hat{\mathcal{R}}_C,\hat{\mathcal{K}}_C)$ admitting an equilibrium $\hat{x}^*\in\mathbb{R}^{n-q}_{> 0}$, the concentration vector $\hat{x}^{\dag}\in\mathbb{R}^{n-q}_{> 0}$ is another equilibrium if and only if
\begin{equation*}
\mathrm{Ln}\left(\frac{\hat{x}^{\dag}}{\hat{x}^{\ast}}\right)\in \mathrm{Ker}(\hat{S}_C^\top).
\end{equation*}
\end{lemma}

\begin{remark}\label{rem:8}
This lemma suggests that if a complex balanced MAS is conservative then it must admit multiple equilibrium points. However, mass conservation law is not necessary to say the existence of multiple equilibria in a complex balanced system. It is still possible to possess multiple equilibria for the complex balanced MAS even if mass conservation does not hold. Of course, whatever the case is, if there are multiple equilibria in a complex balanced MAS, then each equilibrium is in its individual positive stoichiometric compatibility class.
\end{remark}

We reach the following proposition to approach the result that each positive stoichiometric compatibility class in the reverse reconstruction of a MAS $(\mathcal{X,C,R,K})\in\mathscr{N}_\mathrm{Com}$ contains {only} an equilibrium.

\begin{proposition}\label{propo:6}
For a {mass-conserved} MAS $(\mathcal{X,C,R,K})$ admitting an equilibrium $x^*\in\mathbb{R}^n_{\textgreater 0}$, let $(\mathcal{\hat{X}},\hat{\mathcal{C}},\hat{\mathcal{R}},\hat{\mathcal{K}})$ be its reconstruction bridged by the reconstructing matrix $D$ given in \cref{reconstruction matrix}, i.e., satisfying \cref{ReconstructionDefinition1}. Assume that $x_0\in\mathbb{R}^n_{\textgreater 0}$ represents any initial state of $(\mathcal{X,C,R,K})$, then for $\forall~\hat{x}_0\in\mathbb{R}^{n-q}_{\textgreater 0}$ there exists a unique $\mu\in \mathrm{Ker}(\hat{S}^{\top})$ such that
\begin{equation}\label{UEforsub}
  D_{1}[\hat{x}^{\ast}\cdot\mathrm{Exp}(\mu)-\hat{x}_{0}]\in \mathrm{Im}(\hat{S}),
\end{equation}
where $\hat{x}^{\ast}$ is the vector of the first $n-q$ entries in $x^*$.
\end{proposition}

\begin{proof}
Clearly, $\hat{x}^*$ is an equilibrium in $(\mathcal{\hat{X}},\hat{\mathcal{C}},\hat{\mathcal{R}},\hat{\mathcal{K}})$, we thus define a function $\varphi_{D_1}:\mathbb{R}^{n-q}\mapsto\mathbb{R}$ as
\begin{equation}\label{phifunction}
\varphi_{D_{1}}(u)=\sum^{n-q}_{i=1}d_{i}
[\hat{x}_i^*\mathrm{exp}(u_{i})-\hat{x}_{0_i}u_{i}].
\end{equation}
Following the almost completely same proof as in \textit{Appendix B} of
\cite{Feinberg1995} where the function $\sum^{n}_{i=1}[\hat{x}_{i}^{\ast}\mathrm{exp}(u_{i})-\hat{x}_{0_i}u_{i}]$ was proved strictly convex, $\varphi_{D_{1}}(u)$ in \cref{phifunction} can be proved strictly convex and $\lim_{a\rightarrow\infty}\varphi_{D_{1}}(au)=\infty$ for any $u\neq \mathbbold{0}_{n-q}$. Note that the gradient of $\varphi_{D_{1}}(u)$ takes
\begin{equation*}\label{gradient_varphi}
\nabla\varphi_{D_{1}}(u)=D_{1}[\hat{x}^{\ast}\cdot\mathrm{Exp}(u)-\hat{x}_{0}]
\end{equation*}
and $\mathrm{Ker}(\hat{S}^{\top})\subset \mathbb{R}^{n-q}$.
In the same way as in \cite{Feinberg1995}, {we can prove that there exists a unique $\mu\in\mathrm{Ker}(\hat{S}^{\top})$ such that $\nabla\varphi_{D_{1}}^{\top}(\mu)u=0$ for any $u\in\mathrm{Ker}(\hat{S}^{\top})$}, so $\nabla\varphi_{D_{1}}(\mu)\in\mathrm{Im}(\hat{S})$. Namely, for any $\hat{x}_0\in\mathbb{R}^{n-q}_{\textgreater 0}$ there exists a unique $\mu\in \text{Ker}(\hat{S}^{\top})$ such that
\begin{equation*}
  D_{1}[\hat{x}^{\ast}\cdot\mathrm{Exp}(\mu)-\hat{x}_{0}]\in \text{Im}(\hat{S}),
\end{equation*}
which completes the proof.
\end{proof}

{We then give an important result about the reverse reconstruction for a MAS in $\mathscr{N}_\mathrm{Com}$.}

\begin{theorem}\label{gaoReverse}
{Consider a mass-conserved MAS $(\mathcal{X,C,R,K})\in\mathscr{N}_\mathrm{Com}$ governed by \mbox{\cref{CRN-dynamics0}} with an equilibrium $x^*\in\mathbb{R}^n_{\textgreater 0}$. Let $x_{0}\in \mathbb{R}^n_{> 0}$ represent any of its initial state, $(\mathcal{\hat{X}},\hat{\mathcal{C}}_C,\hat{\mathcal{R}}_C,\hat{\mathcal{K}}_C)$ be a complex balanced reconstruction with $D$ given by \mbox{\cref{reconstruction matrix}} as the reconstructing matrix, and $(\mathcal{\tilde{X}},\tilde{\mathcal{C}}_C,\tilde{\mathcal{R}}_C,\tilde{\mathcal{K}}_C)$ be the corresponding reverse reconstruction. Then $\forall~\tilde{x}_{0}\in \mathbb{R}^{n-q}_{> 0}$ there exists a unique equilibrium $\tilde{x}^{\dagger}$ in $(\mathcal{\tilde{X}},\tilde{\mathcal{C}}_C,\tilde{\mathcal{R}}_C,\tilde{\mathcal{K}}_C)$ such that $\tilde{x}^{\dagger}\in \tilde{\mathscr{S}}_C(\tilde{x}_0)\cap \mathbb{R}^{n-q}_{> 0}$. }
\end{theorem}

\begin{proof}
{According to \mbox{\cref{propo:6}}, for any $\hat{x}_{0}=x_{0_\bot}\in \mathbb{R}^{n-q}_{> 0}$ there exists a unique $\mu\in \mathrm{Ker}(\hat{S}_{C}^{\top})$ such that
$D_1[\hat{x}^{\ast}\cdot\mathrm{Exp}(\mu)-\hat{x}_{0}]\in \text{Im}(\hat{S}_{C})$,
which together with \mbox{\cref{cor:3}} yields $[\tilde{x}^{\ast}\cdot\mathrm{Exp}(\mu)-\tilde{x}_{0}]\in \text{Im}(\tilde{S}_{C})$. Note that $\hat{x}^{\ast}$ and $\tilde{x}^{\ast}$ are an equilibrium in $(\mathcal{\hat{X}},\hat{\mathcal{C}}_C,\hat{\mathcal{R}}_C,\hat{\mathcal{K}}_C)$ and $(\mathcal{\tilde{X}},\tilde{\mathcal{C}}_C,\tilde{\mathcal{R}}_C,\tilde{\mathcal{K}}_C)$, respectively. Set $\hat{x}^{\dagger}=\hat{x}^{\ast}\cdot\mathrm{Exp}(\mu)$, which naturally satisfies $\hat{x}^{\dagger}\in \mathbb{R}^{n-q}_{> 0}$, then we get $\mu=\text{Ln}(\frac{\hat{x}^{\dagger}}{\hat{x}^{\ast}})\in\text{Ker}(\hat{S}_{C}^{\top})$. Further based on \mbox{\cref{lem:1}}, we obtain that $\hat{x}^{\dagger}$ is an equilibrium in $(\mathcal{\hat{X}},\hat{\mathcal{C}}_C,\hat{\mathcal{R}}_C,\hat{\mathcal{K}}_C)$, so $\tilde{x}^{\dagger}$ is an equilibrium in $(\mathcal{\tilde{X}},\tilde{\mathcal{C}}_C,\tilde{\mathcal{R}}_C,\tilde{\mathcal{K}}_C)$. Therefore, $\tilde{x}^{\dagger}\in\tilde{\mathscr{S}}_C(\tilde{x}_0)\cap \mathbb{R}^{n-q}_{> 0}$. The uniqueness of $\mu$ guarantees that $\tilde{x}^{\dagger}$ is a unique equilibrium in $\tilde{\mathscr{S}}_C(\tilde{x}_0)\cap \mathbb{R}^{n-q}_{> 0}$.}
\end{proof}

\begin{remark}[Robustness of the first $n-q$ species at equilibrium]\label{rem:9}
{This theorem means that the reverse reconstruction of a mass-conserved MAS $(\mathcal{X,C,R,K})\in\mathscr{N}_\mathrm{Com}$ admitting an equilibrium $x^*\in\mathbb{R}^{n}_{\textgreater 0}$ and starting from any initial sate $x_0\in\mathbb{R}^{n}_{\textgreater 0}$ must contain an equilibrium, and moreover, contains precisely an equilibrium $\tilde{x}^\dag\in\mathbb{R}^{n-q}_{\textgreater 0}$ in each positive stoichiometric compatibility class. However, it should be noted that $\tilde{x}^\dag$ does not guarantee to map an equilibrium in the original system, since from the fact $S_1v(x^\dag)=\tilde{S}\tilde{v}(\tilde{x}^\dag)=\mathbbold{0}_{n-q}$ the last $q$ species in $x^\dag$ can not guarantee to be positive despite being arbitrary. If there exist multiple equilibria in $(\mathcal{X,C,R,K})$ mapped from $\tilde{x}^\dag$, the concentrations of the first $n-q$ species at these different equilibria will be robust with respect to the last $q$ species in the initial state $x_0$, i.e., robust with respect to $x_{0_{\top}}$. Specially, if $\mathrm{dim}\left(\mathrm{Im}(\hat{S}_{C})\right)=n-q$, then $\mu=\mathbbold{0}_{n-q}$. At this time, $x^\dag_{\bot}=x^*_{\bot}$, which indicates the concentrations of the first $n-q$ species will keep unchanged at all of different equilibria, i.e., absolute concentration robustness\mbox{\cite{shinar2010structural}.}}
\end{remark}

\subsection{Asymptotic stability characterized for MASs in $\mathscr{N}_\mathrm{Com}$}
Now we can characterize the local asymptotic stability of any MAS in $\mathscr{N}_\mathrm{Com}$ utilizing the complex balanced reconstruction strategy.

\begin{theorem}\label{thm:4}
For any {mass-conserved} MAS $(\mathcal{X,C,R,K})\in\mathscr{N}_\mathrm{Com}$ described by \mbox{\cref{CRN-dynamics0}}, assume that $(\mathcal{\hat{X}},\hat{\mathcal{C}}_C,\hat{\mathcal{R}}_C,\hat{\mathcal{K}_C})$ is its complex balanced reconstruction with $D$ given by \cref{reconstruction matrix} as the reconstructing matrix, and $(\mathcal{\tilde{X}},\tilde{\mathcal{C}}_C,\tilde{\mathcal{R}}_C,\tilde{\mathcal{K}}_C)$ is the corresponding reverse reconstruction. Let $x^*\in\mathbb{R}^n_{\textgreater 0}$ be an equilibrium in $(\mathcal{X,C,R,K})$, then $\tilde{x}^*\in\mathbb{R}^{n-q}_{\textgreater 0}$ is locally asymptotically stable with respect to all initial conditions in $\tilde{\mathscr{S}}_C(\tilde{x}^*)\cap \mathbb{R}^{n-q}_{> 0}$ near $\tilde{x}^*$.
\end{theorem}

\begin{proof}
{From $S_1v(x^*)=\tilde{S}\tilde{v}(\tilde{x}^*)$, $\tilde{x}^*$ is an equilibrium in $(\mathcal{\tilde{X}},\tilde{\mathcal{C}}_C,\tilde{\mathcal{R}}_C,\tilde{\mathcal{K}}_C)$. }
Further based on \cref{gaoReverse}, $\tilde{x}^*$ is the unique equilibrium in $\tilde{\mathscr{S}}_C(\tilde{x}^*)\cap \mathbb{R}^{n-q}_{> 0}$. For the reverse reconstruction $(\mathcal{\tilde{X}},\tilde{\mathcal{C}}_C,\tilde{\mathcal{R}}_C,\tilde{\mathcal{K}}_C)$ , take the generalized {pseudo-Helmholtz function $\tilde{G}(\tilde{x})$} as the Lyapunov function, which is naturally positive definite with respect to $\tilde{x}-\tilde{x}^*$, as said in \cref{theoremGibbs}. Further, we have its time derivative as
\begin{equation*}
  \dot{\tilde{G}}(\tilde{x})=\nabla ^{\top}\tilde{G}(\tilde{x})\tilde{S}_{C}\tilde{v}_C(\tilde{x})
  =\nabla ^{\top}G(\tilde{x})D_1\tilde{S}_{C}\tilde{v}_C(\tilde{x})=\nabla ^{\top}G(\hat{x})\hat{S}_{C}\hat{v}_{C}(\hat{x}).
\end{equation*}
Note that the pseudo-Helmholtz function $G(\hat{x})$ is a valid Lyapunov function for a complex balanced MAS to suggest the local asymptotic stability at the equilibrium \mbox{\cite{Feinberg1995,Horn72}}, so $\forall~\hat{x}$ we have $\nabla ^{\top}G(\hat{x})\hat{S}_{C}\hat{v}_{C}(\hat{x})\leq 0$ with equality hold at equilibria of the complex balanced reconstruction $(\mathcal{\hat{X}},\hat{\mathcal{C}}_C,\hat{\mathcal{R}}_C,\hat{\mathcal{K}_C})$, i.e., at equilibria of the reverse reconstruction $(\mathcal{\tilde{X}},\tilde{\mathcal{C}}_C,\tilde{\mathcal{R}}_C,\tilde{\mathcal{K}}_C)$.  Since it is impossible for $\tilde{\mathscr{S}}_C(\tilde{x}^*)\cap \mathbb{R}^{n-q}_{> 0}$ to include other equilibrium of the reverse reconstruction,  $\dot{\tilde{G}}(\tilde{x})\leq 0$ with equality hold if and only if $\tilde{x}=\tilde{x}^*$ in this positive stoichiometric compatibility class. Therefore, $\tilde{x}^*$ is locally asymptotically stable with all initial conditions in $\tilde{\mathscr{S}}_C(\tilde{x}^*)\cap \mathbb{R}^{n-q}_{> 0}$ near $\tilde{x}^*$.
\end{proof}

\begin{remark}\label{rem:10}
\cref{thm:4} indicates that for the reverse reconstruction $(\mathcal{\tilde{X}},\tilde{\mathcal{C}}_C,\tilde{\mathcal{R}}_C,\tilde{\mathcal{K}}_C)$, $\forall~\epsilon,\epsilon' \textgreater 0$, $\exists~\delta(\epsilon)\textgreater 0$ such that every solution of $\dot{\tilde{x}}=\tilde{S}_{C}\tilde{v}_{C}(\tilde{x})$ having initial condition within distance $\delta(\epsilon)$, i.e., $\|\tilde{x}_0-\tilde{x}^*\|\textless \delta(\epsilon)$, and $\tilde{x}_0\in\tilde{\mathscr{S}}_C(\tilde{x}^*)\cap \mathbb{R}^{n-q}_{> 0}$, the equilibrium $\tilde{x}^*$ remains within distance $\epsilon$, i.e., $\|\tilde{x}-\tilde{x}^*\|\textless\epsilon$ for all $t\geq 0$ and $\lim_{t\rightarrow{\infty}}\|\tilde{x}(t)-\tilde{x}^*\|\leq \epsilon'$.
\end{remark}

\begin{theorem}\label{thm:5}
For a {mass-conserved} MAS $(\mathcal{X,C,R,K})\in\mathscr{N}_\mathrm{Com}$, let $x^*\in\mathbb{R}^n_{\textgreater 0}$ be any equilibrium, then it is locally asymptotically stable.
\end{theorem}
\begin{proof}
{Assume $x_0\in\mathbb{R}^n_{\textgreater 0}$ to be an initial point for the MAS $(\mathcal{X,C,R,K})$ such that $\tilde{x}_0=x_{0_\bot}\in\tilde{\mathscr{S}}_C(\tilde{x}^*)\cap \mathbb{R}^{n-q}_{> 0}$. From \mbox{\cref{thm:4}} there must exist a neighbourhood of $\tilde{x}^*$ so that if $\tilde{x}_0$ is in this neighbourhood then $\tilde{x}^*$, i.e., $x^*_{\bot}$, is locally asymptotically stable. This means that, based on \mbox{\cref{rem:10}} and \mbox{\cref{propo:4}}, $\forall~\epsilon,\epsilon' \textgreater 0$, $\exists~\delta(\epsilon)\textgreater 0$, when $\|x_{0_{\bot}}-x^*_{\bot}\|\textless \delta(\epsilon)$ we have $\|x_{\bot}-x^*_{\bot}\|\textless\epsilon$ for all $t\geq 0$ and $\lim_{t\rightarrow{\infty}}\|x_{\bot}(t)-x^*_{\bot}\|\leq \epsilon'$. From \mbox{\cref{Dependence}}, we get} $${\|x_{\top}-x^*_{\top}\|\leq\|{C}_{r}^{-\top}{C}_{l}^{\top}\|\cdot \|x_{\bot}-x^*_{\bot}\|.}$$
{Hence,}
$${\|x-x^*\|\leq \sqrt{1+\|{C}_{r}^{-\top}{C}_{l}^{\top}\|^2}\cdot \|x_{\bot}-x^*_{\bot}\|.}$$
{Note that $\sqrt{1+\|{C}_{r}^{-\top}{C}_{l}^{\top}\|^2}$ is a constant, so we have $\forall~\epsilon_1=\sqrt{1+\|{C}_{r}^{-\top}{C}_{l}^{\top}\|^2}\epsilon\textgreater 0$ and $\epsilon_1'=\sqrt{1+\|{C}_{r}^{-\top}{C}_{l}^{\top}\|^2}\epsilon' \textgreater 0$, $\exists~\delta(\epsilon_1)\textgreater 0$, when $\|x_{0}-x^*\|\textless \delta(\epsilon_1)$ we have $\|x-x^*\|\textless\epsilon_1$ for all $t\geq 0$ and $\lim_{t\rightarrow{\infty}}\|x(t)-x^*\|\leq \epsilon_1'$. This implies that the arbitrary equilibrium $x^*$ in $(\mathcal{X,C,R,K})$ is locally asymptotically stable.}
\end{proof}

{Although all of the above results, including the reverse reconstruction possessing an equilibrium in each positive stoichiometric compatibility class, the local asymptotic stability of any MAS in $\mathscr{N}_\mathrm{Com}$, etc., are achieved under the condition of mass conservation in the network, they also apply to the case that mass conservation is not admitted by the network. For the latter, the case will become simpler since the reconstructing matrix is a positive diagonal matrix. There will be no dimension reduction in the reconstruction and the corresponding reverse reconstruction compared to the original network.} In that case, the reconstruction concept has the same signification as the linear conjugacy proposed in \cite{Johnston2011,Johnston2012,Johnston2013}, as stated in \cref{LinearConjugacyComparison}. However, there is a large difference in using Lyapunov function to characterize asymptotic stability of the MAS between these two strategies. The current reconstruction strategy uses the extension of the {pseudo-Helmholtz function} as the Lyapunov function while the latter uses the traditional {pseudo-Helmholtz function}. {Intuitively}, the extension of the {pseudo-Helmholtz function} and the generalized one lay a basis for deeper studies. In the case that the mass conservation law is admitted, these two concepts are completely different. The utilization of the mass conservation law will reduce the complexity of the dynamical model greatly. In this case there will be no definition for the linear conjugacy.

\section{Systematic method to find complex balanced reconstructions and cases studies}\label{sec:5} This section contributes to designing algorithm to compute the complex balanced reconstructions for a MAS in $\mathscr{N}_\mathrm{Com}$, and then illustrating the proposed algorithm through some representative examples.
\subsection{Algorithm for finding complex balanced reconstructions}
Finding complex balanced reconstructions for a MAS $(\mathcal{X,C,R,K})\in\mathscr{N}_\mathrm{Com}$ means to design $D,\hat{\mathcal{C}}_{C},\hat{\mathcal{R}}_{C}$ and $\hat{\mathcal{K}}_{C}$ such that
\begin{equation}\label{ReconstructionCondition}
\begin{cases}
  D_1S_1v(x)=\hat{S}_{C}\hat{v}_{C}(\hat{x}),\\
  \hat{B}_{C}\hat{v}_{C}(\hat{x}^{\ast})=\mathbbold{0}_{\hat{c}},
\end{cases}
\end{equation}
where $\hat{B}_{C}$ and $\hat{x}^{\ast}$ are the incident matrix and complex balanced equilibrium in the complex balanced MAS $(\mathcal{\hat{X}},\hat{\mathcal{C}}_C,\hat{\mathcal{R}}_C,\hat{\mathcal{K}}_C)$, respectively. Every MAS with the dynamics $\dot{\hat{x}}=\hat{S}_{C}\hat{v}_{C}(\hat{x})$ is a reconstruction needed to be found. Essentially, it is to find realizations with a complex balancing structure for the nonnegative autonomous polynomial system with the vector field of $D_1S_1v(x)$. Therefore, the known algorithms for finding realizations \cite{Szederkenyi2010,Szederkenyi2011a,Szederkenyi2011b,Johnston2013} naturally constitute the base to achieve complex balanced reconstructions.

Note that the dynamics $\dot{x}=Sv(x)$ is expressed in a reaction centered formulation, which is structure-hidden. To reflect the underlying CRN structure, we use complex centered formulation \cite{Szederkenyi2011a} instead of the current one through substituting \cref{Stoichiomatrix} into the dynamical equation, and then replacing $Bv(x)$ by $L\varPsi(x)$, where
$L\in\mathbb{R}^{c\times c}$ is the Kirchhoff matrix that stores the reaction rate coefficients according to \begin{equation}\label{Kirchhoff}
L_{\rho\pi}= \begin{cases}
k_{\pi\rho} & \rho\neq \pi,\\
-\sum_{l\neq \pi}^{c}L_{l\pi} & \rho=\pi,
\end{cases}
\end{equation}
with $k_{\pi\rho}(\rho,\pi=1,\cdots,c)$ indicating the reaction rate coefficient of reaction $Z_{\cdot \pi}\to Z_{\cdot \rho}$, and $\varPsi(x)\in \mathbb{R}^c$ is a polynomial vector with the $\rho$th entry to be
\begin{equation}\label{Psai}
\varPsi_{\rho}(x)=\prod_{i=1}^{n}x^{Z_{i\rho}}.
\end{equation}
Clearly, $L$ is column conservative, constrained by $\mathbbold{1}_{c}^\top L=\mathbbold{0}_{c}^\top$, and $\varPsi(x)$ can be solely identified by $Z$. In \cref{Kirchhoff}, $k_{\pi\rho}=0$ when $Z_{\cdot \pi}\to Z_{\cdot \rho}$ is not existing. Note that these reaction rate coefficients are identified by two complexes, which are essentially the same with those in \cref{ReactionRate} identified by reactions. Utilizing the complex centered formulation, we can rewrite the complex balanced reconstruction condition \cref{ReconstructionCondition} as
\begin{equation}\label{ReconstructionCondition1}
\begin{cases}
{DZL\varPsi(x)=}\begin{bmatrix}{\hat{Z}_C\hat{L}_C\hat{\varPsi}_C(\hat{x})}\\ {\mathbbold{0}_{q}}\end{bmatrix},\\
{\hat{L}_C\hat{\varPsi}_C(\hat{x}^*)=\mathbbold{0}_{\hat{c}},}\\{x_{\top}=C_{r}^{-\top}C^{\top}_l(\hat{x}^*-\hat{x})+x_{\top}^*,}
\end{cases}
\end{equation}
{where the third constraint serves for eliminating those non-free species in the complex balanced reconstruction.} Therefore, the essence of finding complex balanced reconstructions is to find complex balanced realizations, uniquely characterized by the pair $(\hat{Z}_C,\hat{L}_C)$, for the nonnegative autonomous polynomial system with the vector field of $DZL\varPsi(x)$ under given $D$.

There exist a few optimization based algorithms to find realizations, like minimizing/maximizing the sum of reaction rate coefficients \cite{Szederkenyi2011a}, minimizing/maximizing the number of reactions
\cite{Szederkenyi2010,Szederkenyi2011b}, minimizing the deficiency of weakly reversible realizations
\cite{Johnston2013}, etc. Some linear programming and mixed-integer linear programming problems are designed towards these objects. Here, for simplicity we use the strategy of minimizing the sum of reaction rate coefficients to design algorithm for finding complex balanced reconstructions. In addition, we calculate the conserved matrix $C$ for the original MAS $(\mathcal{X,C,R,K})$ beforehand, which may be easily obtained through solving the linear system of equations $S^\top\xi=\mathbbold{0}_r$, and then following \cref{defConservedMatrix}.

With the above information, we can design the linear programming problem towards solving complex balanced reconstructions for the kinetic system $\dot{x}=ZL\varPsi(x)$ to be
\begin{alignat}{1}
\min_{D_1, \hat{L}_{{C}}} & ~~\sum_{\rho=1}^{\hat{c}}\sum_{\pi\neq \rho}^{\hat{c}}\hat{L}_{{C}_{\rho\pi}}   \tag{{P1}} \label{LP for complex reconstruction}\\
 \mbox{s.t.}
 & ~\begin{cases}\nonumber
{DZL\varPsi(x)=}\begin{bmatrix}{\hat{Z}_C\hat{L}_C\hat{\varPsi}_C(\hat{x})}\\ {\mathbbold{0}_{q}}\end{bmatrix},  \\{x_{\top}=C_{r}^{-\top}C^{\top}_l(\hat{x}^*-\hat{x})+x_{\top}^*,}\\
 {\mathbbold{1}_{\hat{c}}^{\top}\hat{L}_{C}=\mathbbold{0}_{\hat{c}}^{\top},} \\
 {\hat{L}_{{C}_{\rho\pi}}\geq 0, ~~~ \forall \rho \neq \pi, ~~ \rho,\pi=1,\cdots,\hat{c},}\\
 {\hat{L}_{C}\hat{\varPsi}_{C}(\hat{x}^{\ast})=\mathbbold{0}_{\hat{c}},}\\
 {\epsilon\leq d_{i}\leq\frac{1}{\epsilon}, ~~~i=1,\cdots,n-q,}
\end{cases}
\end{alignat}
where $D_1=\mathrm{diag}(d_i)$ and $\hat{L}_{C}$ are the decision variables, and $\epsilon> 0$ is a given very small positive number. Once $Z$ and $L$ are given, then $x^*$, i.e., $\hat{x}^*=x^*_{\bot}$, may be calculated based on $ZL{\varPsi}(x^{\ast})=\mathbbold{0}_n$. These parameters together with a given $\epsilon$ are known data for \cref{LP for complex reconstruction}. The pair $(\hat{Z}_C,\hat{L}_C)$ solved from \cref{LP for complex reconstruction}, namely, represents the desired complex balanced reconstruction.

\begin{table*}\tiny
\centering \caption{Some mass action systems and their complex balanced reconstructions} \label{table:1}
\begin{tabular}{|m{10pt}<{\centering}|m{115pt}<{\centering}|m{36pt}<{\centering}|m{115pt}<{\centering}|m{75pt}<{\centering}|}
\hline
\tabincell{c}{No.} & \tabincell{c}{MAS\\$(\mathcal{X,C,R,K})$} & \tabincell{c}{Equilibrium\\$x^*$ }  & \tabincell{c}{Complex Balanced Reconstruction\\$(\mathcal{X},\hat{\mathcal{C}}_C,\hat{\mathcal{R}}_C,\hat{\mathcal{K}}_C)$}  & \tabincell{c}{Reconstructing Matrix\\$D$} \\
\hline\hline
1
&
\begin{minipage}{4cm} \vspace{-0.0cm}\centering
\multirow{2}*{\tabincell{c}{$\xymatrix{ 2X_{1} \ar ^{1~~} [r] & 2X_{2} }$\\
$\xymatrix{X_{2} \ar ^{1~~} [r] & X_{1}}$}}
 \vspace{0.8cm}\end{minipage}
&
\begin{minipage}{1.1cm} \vspace{0.1cm}\centering
$\begin{bmatrix}
1 \\
2 \\
\end{bmatrix}$
 \vspace{0.05cm}\end{minipage}
&
\begin{minipage}{4cm} \vspace{0.1cm}\centering
$\xymatrix{X_{1} \ar @{ -^{>}}^{0.01~}  @< 1pt> [r] & \varnothing^\S  \ar @{ -^{>}}^{0.008~} @< 1pt> [l] \ar @{ -^{>}}^{0.011~~}  @< 1pt> [r] & 2X_{1} \ar @{ -^{>}}^{0.009~~} @< 1pt> [l]}$
\vspace{0.2cm}\end{minipage}
&
$\begin{bmatrix}
0.01 & 0  \\
1& 1
\end{bmatrix}$\\
\hline
2
&
\begin{minipage}{4cm} \vspace{0.1cm}\centering
\xymatrix{  2X_1  \ar @{ -^{>}}^{1~~~~~}  @< 1pt> [d]\\
  X_{1}+X_{2} \ar @{ -^{>}}^{~2}  @< 1pt> [u] \ar ^{~~1}  @< 1pt> [r]  &2X_2}
\vspace{0.2cm}\end{minipage}
 \begin{minipage}{4cm} \vspace{0.3cm}\centering
\vspace{0.2cm}\end{minipage}
&
\begin{minipage}{1.1cm} \vspace{0.1cm}\centering
$\begin{bmatrix}
1 \\
1 \\
\end{bmatrix}$
 \vspace{0.05cm}\end{minipage}
&
\begin{minipage}{4cm} \vspace{0.8cm}\centering
\multirow{2}*{\tabincell{c}{$\xymatrix{2X_{1} \ar @{ -^{>}}^{0.02~}  @< 1pt> [r]   & X_1 \ar @{ -^{>}}^{0.02~} @< 1pt> [l] }$ }}
\vspace{1.3cm}\end{minipage}
& $\begin{bmatrix}
0.01 & 0  \\
1 & 1
\end{bmatrix}$  \\
\hline
3
&
\begin{minipage}{4cm} \vspace{0.2cm}\centering
\multirow{3}*{\tabincell{c}{$\xymatrix{ X_{1}+X_2 \ar ^{~~~9} [r] & X_{3} }$\\
\xymatrix{X_{1} \ar @{ -^{>}}^{7}  @< 1pt> [r]  \ar @{ -^{>}}^{1}  @< 1pt> [rd] & X_2 \ar @{ -^{>}}^{10} @< 1pt> [l] \ar @{ -^{>}}^{1~~~~~~}  @< 1pt> [d]\\
 & X_{3} \ar @{ -^{>}}^{7}  @< 1pt> [u]  \ar @{ -^{>}}^{1}  @< 1pt> [lu]}
}}
 \vspace{2.2cm}\end{minipage}
&
\begin{minipage}{1.1cm} \vspace{0.1cm}\centering
$\begin{bmatrix}
\frac{1}{3} \\
\frac{1}{3} \\
\frac{1}{3}\\
\end{bmatrix}$
 \vspace{0.05cm}\end{minipage}
&
\begin{minipage}{4cm} \vspace{0.2cm}\centering
\multirow{3}*{\tabincell{c}{\xymatrix{ X_{1}+X_2  \ar @{ -^{>}}^{~~~~0.09}  @< 1pt> [r]   & \varnothing\ar @{ -^{>}}^{~~~~0.01}  @< 1pt> [l]}\\\xymatrix{X_{1} \ar _{0.09}  [rd] & X_2 \ar _{~~~0.09} [l]    \ar @{ -^{>}}^{0.09~}  @< 1pt> [d]\\
 & \varnothing \ar @{ -^{>}}^{0.06}  @< 1pt> [u] }
}}
\vspace{2.2cm}\end{minipage}
&
$\begin{bmatrix}
0.01& 0 & 0 \\
0 & 0.01 & 0 \\
1 & 1 & 1
\end{bmatrix}$\\
\hline
4
&
\begin{minipage}{4cm} \vspace{0cm}\centering
\multirow{2}*{\tabincell{c}{$\xymatrix{X_{1}+X_{2} \ar ^{~~~~2} [r] & X_{3} }$\\$\xymatrix{2X_{3} \ar ^{1~~~~~} [r] &2X_{1}+2X_{2}}$
}}
 \vspace{0.9cm}\end{minipage}
&
\begin{minipage}{1.1cm} \vspace{0.1cm}\centering
$\begin{bmatrix}
0.5 \\
0.5 \\
0.5 \\
\end{bmatrix}$
 \vspace{0.05cm}\end{minipage}
&
\begin{minipage}{4cm} \vspace{0.4cm}\centering
\xymatrix{ X_3  \ar @{ -^{>}}^{0.04}  @< 1pt> [r]   & \varnothing\ar @{ -^{>}}^{0.02}  @< 1pt> [l]}
\vspace{0.4cm}\end{minipage}
& $\begin{bmatrix}
0.01 & 0 & 0 \\
1 & 1 & 2 \\
1 & 2 & 3
\end{bmatrix}$  \\
\hline
5
&
\begin{minipage}{4cm} \vspace{0.3cm}\centering
\multirow{4}*{\tabincell{c}{$\xymatrix{4X_{3} \ar ^{3.1~~~~} [r] & X_{2}+2X_{3} }$ \\
$\xymatrix{2X_{2} \ar ^{5~~~~} [r] & X_{1}+X_{2}}$\\
$\xymatrix{X_{4} \ar ^{2.8~} [r] & X_{1} \ar ^{10.6~~} [r] & 2X_{3} }$ \\
$\xymatrix{X_{1}+X_{4} \ar ^{~~9.1} [r] & 2X_{4}}$ }}
 \vspace{2.0cm}\end{minipage}
&
\begin{minipage}{1.1cm} \vspace{0.1cm}\centering
$\begin{bmatrix}
0.3077 \\
0.8077 \\
1.0128 \\
0.6
\end{bmatrix}$
 \vspace{0.05cm}\end{minipage}
&
\begin{minipage}{4cm} \vspace{0.3cm}\centering
\multirow{2}*{\tabincell{c}{$\xymatrix{4X_{3} \ar ^{3.1} [r]   & 2X_{2} \ar ^{~5} [d]\\
 & X_{1}   \ar ^{10.6~~} [lu]} $ \\
$\xymatrix{X_{4} \ar @{ -^{>}}^{2.8~~~}  @< 1pt> [r]   & X_{1}+X_{4} \ar @{ -^{>}}^{9.1~~~} @< 1pt> [l] }$ }}
\vspace{2.5cm}\end{minipage}
& $\begin{bmatrix}
1 & 0 & 0 & 0\\
0 & 2 & 0 & 0\\
0 & 0 & 2 & 0\\
2 & 2 & 1 & 2
\end{bmatrix}$  \\
\hline
6
&
\begin{minipage}{4cm} \vspace{0.5cm}\centering
\multirow{3}*{\tabincell{c}{
\xymatrix{X_{1} \ar @{ -^{>}}^{1}  @< 1pt> [r]   & X_{3} \ar @{ -^{>}}^{1}  @< 1pt> [l] }\\\xymatrix{X_2 \ar @{ -^{>}}^{1}  @< 1pt> [r]   & X_4 \ar @{ -^{>}}^{1}  @< 1pt> [l] }\\
$\xymatrix{X_{1}+X_{2} \ar ^{1~~} [r] & X_{3}+X_{4} }$}}
 \vspace{2.2cm}\end{minipage}
&
\begin{minipage}{1.1cm} \vspace{0.1cm}\centering
$\begin{bmatrix}
1  \\
1  \\
2\\2
\end{bmatrix}$
\vspace{0.05cm}\end{minipage}
&
\begin{minipage}{2cm} \vspace{0.1cm}\centering
\xymatrix{&X_1+X_{2}\ar @{ -^{>}}^{0.01}  @< 1pt> [d]\\X_1\ar @{ -^{>}}^{0.02}  @< 1pt> [r]& \varnothing \ar @{ -^{>}}^{0.01}  @< 1pt> [u]\ar @{ -^{>}}^{0.02}  @< 1pt> [l]   \ar @{ -^{>}}^{0.02}  @< 1pt> [r] & X_{2} \ar @{ -^{>}}^{0.02}  @< 1pt> [l]}
\vspace{0.2cm}\end{minipage}
&
$\begin{bmatrix}
0.01 & 0 & 0&0  \\
0 & 0.01 & 0 &0 \\
1 & 1 & 1&1\\
1&2&1&2
\end{bmatrix}$ \\
\hline

 \multicolumn{3}{l}{\hspace{-0.2cm}\vspace{-0.2cm}\scriptsize{$^\S$ $\varnothing $ represents zero complex.}}
\end{tabular}
\end{table*}

\subsection{Cases studies}
In this subsection, some examples of MASs and their complex balanced reconstructions computed through solving the linear programming \cref{LP for complex reconstruction} are shown. \cref{table:1} reports the computation results of complex balanced reconstructions for $6$ MASs. The first two examples are two-species CRNs, and the third and fourth examples are three-species but having the conserved matrices of rank $1$ and $2$, respectively. The fifth and sixth CRNs contain four species. Clearly, every MAS has a complex balanced reconstruction, and for the first, second, fourth and sixth examples the complex balanced reconstructions are even detailed balanced. The reconstructing matrix in each case is also reported in \cref{table:1}.

\section{Conclusions}\label{sec:6}
This paper has developed the generalized {pseudo-Helmholtz function} for stability analysis for general MASs, which fully utilizes the fact that the {pseudo-Helmholtz function} is a valid Lyapunov function for characterizing local asymptotic stability for complex balanced MASs. To match the use of the generalized {pseudo-Helmholtz function} as the Lyapunov function, the notions of reconstruction and reverse reconstruction for the original network have been defined in succession through an invertible matrix storing a positive diagonal matrix and the conserved matrix that captures mass conservation of the original network. These two kinds of defined networks will retain some properties of the original one, such as their equilibrium obtained from the concentrations of the first $n-q$ species of the equilibrium in the original network, the dynamics of the reverse reconstruction bridged to that of the original network by a constant matrix, etc. We have subsequently proved that if the original network has a complex balanced reconstruction, then it is stable. The asymptotic stability of the original network has been further reached based on the complex balanced reconstruction strategy in which the corresponding reverse reconstruction is proved to possess only an equilibrium in each positive stoichiometric compatibility class, and thus to be asymptotically stable. To facilitate applications of the proposed reconstruction strategy, a constructive approach is proposed to compute complex balanced reconstructions for general MASs.

\section*{Acknowledgments}
We would like to acknowledge anonymous reviewers for their valuable comments, especially on the definition of reconstruction. Thanks are also given to Dr. Xiaoyu Zhang for her computation on examples in Table 1, and to the National Natural Science Foundation of China under Grant Nos. 11671418 and 61611130124 for financial support.


\begin{thebibliography}{10}
	
	\bibitem{Al-Radhawi2016}
	{\sc M.~A. Al-Radhawi and D.~Angeli}, \textit{New approach to the stability of chemical reaction networks: piecewise linear in rates lyapunov functions},
	IEEE Trans. Automa. Control, 61 (2016), pp.~76--89.
	
	\bibitem{Anderson2011}
	{\sc D.~F. Anderson}, \textit{A proof of the global attractor conjecture in the
		single linkage class case}, SIAM J. Appl. Math., 71 (2011), pp.~1487--1508.
	
	\bibitem{Anderson15}
	{\sc D.~F. Anderson, G.~Craciun, M.~Gopalkrishnan, and C.~Wiuf}, \textit{Lyapunov
		functions, stationary distributions, and non-equilibrium potential for
		reaction networks}, Bull. Math. Biol., 77 (2015),
	pp.~1744--1767.
	
	\bibitem{Angeli2007}
	{\sc D.~Angeli, P.~De~Leenheer, and E.~D. Sontag}, \textit{A petri net approach to
		the study of persistence in chemical reaction networks}, Math. Bios., 210
	(2007), pp.~598--618.
	
	\bibitem{Angeli2011}
	{\sc D.~Angeli, P.~De~Leenheer, and E.~D. Sontag}, \textit{Persistence resutls for
		chemical reaction networks with time-dependence kinetics and no global
		conservation laws}, SIAM J. Math. Anal., 71 (2011), pp.~128--146.
	
	
	
	\bibitem{Craciun2006}
	{\sc G.~Craciun and M.~Feinberg}, \textit{Multiple equilibria in complex chemical
		reaction networks: \uppercase\expandafter{\romannumeral2}. the
		species-reaction graph}, SIAM J. Appl. Math., 66 (2006), pp.~1321--1338.
	
	\bibitem{Craciun2013}
	{\sc G.~Craciun, F.~Nazarov, and C.~Pantea}, \textit{Persistence and permanence of
		mass-action and power-law dynamical systems}, SIAM J. Appl. Math., 73 (2013),
	pp.~305--329.
	
	\bibitem{Donnell13}
	{\sc P.~Donnell and M.~Banaji}, \textit{Local and global stability of equilibria
		for a class of chemical reaction networks}, SIAM J. Applied Dynmical Systems,
	12 (2013), pp.~899--920.
	
	\bibitem{Feinberg87}
	{\sc M.~Feinberg}, \textit{Chemical reaction network structure and the stability
		of complex isothermal reactors-\uppercase\expandafter{\romannumeral1}. the
		deficiency zero and deficiency one theorems}, Chem. Eng. Sci.,
	42 (1987), pp.~2229--2268.
	
	\bibitem{Feinberg1995}
	{\sc M.~Feinberg}, \textit{The existence and uniqueness of steady states for a
		class of chemical reaction networks}, Arch. Rati. Mech. Anal., 132 (1995),
	pp.~311--370.
	
	
	\bibitem{Toth81}
	{\sc V.~H\'{a}rs and J.~T\'{o}th}, \textit{In qualitative theory of differential
		equations}, ed. by M. Farkas, L. Hatvani. On the Inverse Problem of Reaction
	Kinetics, Coll. Math. Soc. J. Bolyai (North-Holland, Amsterdam), 30 (1981),
	pp.~363--379.
	
	\bibitem{Hethcote2000}
	{\sc H.~W. Hethcote}, \textit{The mathematics of infectious diseases}, SIAM
	Review, 42 (2000), pp.~599--653.
	
	\bibitem{Horn72}
	{\sc F.~J.~M. Horn and R.~Jackson}, \textit{General mass action kinetics}, Arch.
	Rati. Mech. Anal., 47 (1972), pp.~81--116.
	
	\bibitem{Johnston2011}
	{\sc M.~D. Johnston and D.~Siegel}, \textit{Linear conjugacy of chemical reaction
		networks}, J. Math. Chem., 49 (2011), pp.~1263--1282.
	
	\bibitem{Johnston2012}
	{\sc M.~D. Johnston, D.~Siegel, and G.~Szederk\'{e}nyi}, \textit{Dynamical
		equivalence and linear conjugacy of chemical reaction networks: new results
		and methods}, Match-Comm. Math. Comp. Chem., 68 (2012), pp.~443--468.
	
	\bibitem{Johnston2013}
	{\sc M.~D. Johnston, D.~Siegel, and G.~Szederk\'{e}nyi}, \textit{Computing weakly
		reversible linearly conjugate chemical reaction networks with minimal
		deficiency}, Math. Bios., 241 (2013), pp.~88--98.
	
	\bibitem{Joshi12}
	{\sc B.~Joshi and A.~Shiu}, \textit{Simplifying the jocobian criterion for
		precluding multistationarity in chemical reaction networks}, SIAM J. Appl.
	Math., 72 (2012), pp.~857--876.
	
	\bibitem{Minke2017}
	{\sc M.~Ke, S.~Wu, and C.~H. Gao}, \textit{Realizations of quasi-polynomial
		systems and application for stability analysis}, J. Math. Chem., 55 (2017),
	pp.~1597--1621.
	
	\bibitem{Mason2009}
	{\sc O.~Mason, V.~S. Bokharaie, and R.~Shorten}, \textit{Stability and d-stability
		for switched positive systems}, Lecture Notes in Control and Information
	Sciences, 389 (2009), pp.~101--109.
	
	\bibitem{Pantea2012}
	{\sc C.~Pantea}, \textit{On the persistence and global stability of mass-action
		systems}, SIAM J. Math. Anal., 44 (2012), pp.~1636--1673.
	
	\bibitem{Rao13}
	{\sc S.~Rao, A.~van~der Schaft, and B.~Jayawardhana}, \textit{A graph-theoretical
		approach for the analysis and model reduction of complex-balanced chemical
		reaction networks}, J. Math. Chem., 51 (2013), pp.~2401--2422.
	
	\bibitem{shinar2010structural}
	{\sc G.~Shinar and M.~Feinberg}, \textit{Structural sources of robustness in
		biochemical reaction networks}, Science, 327 (2010), pp.~1389--1391.
	
	\bibitem{Sontag2001}
	{\sc E.~D. Sontag}, \textit{Structure and stability of certain chemical networks
		and applications to the kinetic proofreading model of t-cell receptor signal
		transduction}, IEEE Trans. Automa. Control, 46 (2001), pp.~1028--1047.
	
	\bibitem{Szederkenyi2010}
	{\sc G.~Szederk\'{e}nyi}, \textit{Computing sparse and dense realizations of
		reaction kinetic systems}, J. Math. Chem., 47 (2010), pp.~551--568.
	
	\bibitem{Szederkenyi2011a}
	{\sc G.~Szederk\'{e}nyi and K.~M. Hangos}, \textit{Finding complex balanced and
		detailed balanced realizations of chemical reaction networks}, J. Math.
	Chem., 49 (2011), pp.~1163--1179.
	
	\bibitem{Szederkenyi2011b}
	{\sc G.~Szederk\'{e}nyi, K.~M. Hangos, and T.~P\'{e}ni}, \textit{Maximal and
		minimal realizations of reaction kinetic systems: Computation and
		properties}, Match Comm. Math. Comp. Chem., 65 (2011), pp.~309--332.
	
	\bibitem{Schaft13}
	{\sc A.~van~der Schaft, S.~Rao, and B.~Jayawardhana}, \textit{On the mathematical
		structure of balanced chemical reaction networks governed by mass action
		kinetics}, SIAM J. Appl. Math., 73 (2013), pp.~953--973.
	
\end{thebibliography}
\end{document}